\documentclass[reqno, 12pt]{amsart}
\pdfoutput=1
\makeatletter
\let\origsection=\section \def\section{\@ifstar{\origsection*}{\mysection}}
\def\mysection{\@startsection{section}{1}\z@{.7\linespacing\@plus\linespacing}{.5\linespacing}{\normalfont\scshape\centering\S}}
\makeatother

\usepackage{amsmath,amssymb,amsthm}
\usepackage{mathrsfs}
\usepackage{mathabx}\changenotsign
\usepackage{dsfont}

\usepackage{xcolor}
\usepackage[backref]{hyperref}
\hypersetup{
    colorlinks,
    linkcolor={red!60!black},
    citecolor={green!60!black},
    urlcolor={blue!60!black}
}

\usepackage[open,openlevel=2,atend]{bookmark}

\usepackage[abbrev,msc-links,backrefs]{amsrefs}
\usepackage{doi}

\renewcommand{\PrintDOI}[1]{\doi{#1}}

\usepackage[T1]{fontenc}
\usepackage{lmodern}
\usepackage[babel]{microtype}
\usepackage[english]{babel}

\usepackage{geometry}
\numberwithin{equation}{section}
\numberwithin{figure}{section}

\usepackage{enumitem}
\def\rmlabel{\upshape({\itshape \roman*\,})}

\def\alabel{\upshape({\itshape \alph*\,})}

\def\nlabel{\upshape({\itshape \arabic*\,})}

\let\polishlcross=\l
\def\l{\ifmmode\ell\else\polishlcross\fi}

\def\qand{\quad\text{and}\quad}
\def\qqand{\qquad\text{and}\qquad}

\let\emptyset=\varnothing
\let\setminus=\smallsetminus

\makeatletter
\def\moverlay{\mathpalette\mov@rlay}
\def\mov@rlay#1#2{\leavevmode\vtop{   \baselineskip\z@skip \lineskiplimit-\maxdimen
   \ialign{\hfil$\m@th#1##$\hfil\cr#2\crcr}}}
\newcommand{\charfusion}[3][\mathord]{
    #1{\ifx#1\mathop\vphantom{#2}\fi
        \mathpalette\mov@rlay{#2\cr#3}
      }
    \ifx#1\mathop\expandafter\displaylimits\fi}
\makeatother

\newcommand{\dcup}{\charfusion[\mathbin]{\cup}{\cdot}}

\DeclareFontFamily{U}  {MnSymbolC}{}
\DeclareSymbolFont{MnSyC}         {U}  {MnSymbolC}{m}{n}
\DeclareFontShape{U}{MnSymbolC}{m}{n}{
    <-6>  MnSymbolC5
   <6-7>  MnSymbolC6
   <7-8>  MnSymbolC7
   <8-9>  MnSymbolC8
   <9-10> MnSymbolC9
  <10-12> MnSymbolC10
  <12->   MnSymbolC12}{}
\DeclareMathSymbol{\powerset}{\mathord}{MnSyC}{180}

\usepackage{tikz}
\usetikzlibrary{calc,decorations.pathmorphing}
\pgfdeclarelayer{background}
\pgfdeclarelayer{foreground}
\pgfdeclarelayer{front}
\pgfsetlayers{background,main,foreground,front}

\usepackage{subcaption}
\newcommand{\qedge}[7]{

	\ifx\relax#4\relax
		\def\qoffs{0pt}
	\else
		\def\qoffs{#4}
	\fi

	\def\qhedge{
		($#1+#3!\qoffs!-90:#2-#3$) --
		($#2+#1!\qoffs!-90:#3-#1$) --
		($#3+#2!\qoffs!-90:#1-#2$) -- cycle}

	\coordinate (12) at ($#1!\qoffs!90:#2$);
	\coordinate (13) at ($#1!\qoffs!-90:#3$);
	\coordinate (23) at ($#2!\qoffs!90:#3$);
	\coordinate (21) at ($#2!\qoffs!-90:#1$);
	\coordinate (31) at ($#3!\qoffs!90:#1$);
	\coordinate (32) at ($#3!\qoffs!-90:#2$);
	
	\def\nqhedge{
		(13) let \p1=($(13)-#1$), \p2=($(12)-#1$) in
			arc[start angle={atan2(\y1,\x1)}, delta angle={atan2(\y2,\x2)-atan2(\y1,\x1)-360*(atan2(\y2,\x2)-atan2(\y1,\x1)>0)}, x radius=\qoffs, y radius=\qoffs] --
		(21) let \p1=($(21)-#2$), \p2=($(23)-#2$) in
			arc[start angle={atan2(\y1,\x1)}, delta angle={atan2(\y2,\x2)-atan2(\y1,\x1)-360*(atan2(\y2,\x2)-atan2(\y1,\x1)>0)}, x radius=\qoffs, y radius=\qoffs] --
		(32) let \p1=($(32)-#3$), \p2=($(31)-#3$) in
			arc[start angle={atan2(\y1,\x1)}, delta angle={atan2(\y2,\x2)-atan2(\y1,\x1)-360*(atan2(\y2,\x2)-atan2(\y1,\x1)>0)}, x radius=\qoffs, y radius=\qoffs] --
		cycle}

		\ifx\relax#5\relax
		\def\qlwidth{1pt}
	\else
		\def\qlwidth{#5}
	\fi
	
		\ifx\relax#7\relax
		\fill \nqhedge;
	\else
		\fill[#7]\nqhedge;
	\fi

		\ifx\relax#6\relax
		\draw[line width=\qlwidth,rounded corners=\qoffs]\nqhedge;
	\else
		\draw[line width=\qlwidth,#6]\nqhedge;
	\fi
}

\let\epsilon=\varepsilon
\let\eps=\epsilon
\let\rho=\varrho
\let\theta=\vartheta
\let\wh=\widehat

\def\EE{{\mathds E}}
\def\NN{{\mathds N}}

\def\PP{{\mathds P}}
\def\RR{{\mathds R}}

\newcommand{\cA}{\mathcal{A}}

\newcommand{\cF}{\mathcal{F}}
\newcommand{\cG}{\mathcal{G}}

\newcommand{\cR}{\mathcal{R}}
\newcommand{\cS}{\mathcal{S}}

\newcommand{\cX}{\mathcal{X}}

\newcommand{\ccB}{\mathscr{B}}

\newcommand{\ccC}{\mathscr{C}}
\newcommand{\ccP}{\mathscr{P}}
\newcommand{\gS}{\mathfrak{S}}

\newtheoremstyle{note}  {4pt}  {4pt}  {\sl}  {}  {\bfseries}  {.}  {.5em}          {}
\newtheoremstyle{introthms}  {3pt}  {3pt}  {\itshape}  {}  {\bfseries}  {.}  {.5em}          {\thmnote{#3}}
\newtheoremstyle{remark}  {2pt}  {2pt}  {\rm}  {}  {\bfseries}  {.}  {.3em}          {}

\theoremstyle{plain}
\newtheorem{theorem}{Theorem}[section]
\newtheorem{lemma}[theorem]{Lemma}

\newtheorem{prop}[theorem]{Proposition}

\newtheorem{cor}[theorem]{Corollary}
\newtheorem{fact}[theorem]{Fact}
\newtheorem{claim}[theorem]{Claim}

\theoremstyle{note}
\newtheorem{dfn}[theorem]{Definition}
\newtheorem{setup}[theorem]{Setup}

\theoremstyle{remark}

\newtheorem{exmpl}[theorem]{Example}

\usepackage{accents}
\newcommand{\seq}[1]{\accentset{\rightharpoonup}{#1}}

\usepackage{lineno}
\newcommand*\patchAmsMathEnvironmentForLineno[1]{\expandafter\let\csname old#1\expandafter\endcsname\csname #1\endcsname
\expandafter\let\csname oldend#1\expandafter\endcsname\csname end#1\endcsname
\renewenvironment{#1}{\linenomath\csname old#1\endcsname}{\csname oldend#1\endcsname\endlinenomath}}\newcommand*\patchBothAmsMathEnvironmentsForLineno[1]{\patchAmsMathEnvironmentForLineno{#1}\patchAmsMathEnvironmentForLineno{#1*}}\AtBeginDocument{\patchBothAmsMathEnvironmentsForLineno{equation}\patchBothAmsMathEnvironmentsForLineno{align}\patchBothAmsMathEnvironmentsForLineno{flalign}\patchBothAmsMathEnvironmentsForLineno{alignat}\patchBothAmsMathEnvironmentsForLineno{gather}\patchBothAmsMathEnvironmentsForLineno{multline}}

\def\odd{\text{\rm odd}}
\def\even{\text{\rm even}}
\def\flex{\text{\rm flex}}

\def\Ubad{U_\text{\rm bad}}

\begin{document}

\title[Minimum degree condition for tight Hamiltonian cycles in hypergraphs]{Minimum vertex degree condition for tight Hamiltonian cycles in $3$-uniform hypergraphs}

\dedicatory{Dedicated to the memory of Andr\'as Hajnal}

\author[Chr.~Reiher]{Christian Reiher}
\address{Fachbereich Mathematik, Universit\"at Hamburg, Hamburg, Germany}
\email{Christian.Reiher@uni-hamburg.de}
\email{schacht@math.uni-hamburg.de}

\author[V.~R\"{o}dl]{Vojt\v{e}ch R\"{o}dl}
\address{Department of Mathematics and Computer Science,
Emory University, Atlanta, USA}
\email{rodl@mathcs.emory.edu}
\thanks{The second author is supported by NSF grant DMS 1301698.}

\author[A.~Ruci\'nski]{Andrzej Ruci\'nski}
\address{A. Mickiewicz University, Department of Discrete Mathematics, Pozna\'n, Poland}
\email{rucinski@amu.edu.pl}
\thanks{The third author is supported by Polish NSC grant 2014/15/B/ST1/01688.}

\author[M.~Schacht]{Mathias Schacht}

\author[E.~Szemer\'edi]{Endre Szemer\'edi}
\address{Alfr\'ed R\'enyi Institute of Mathematics,
	 Hungarian Academy of Sciences,
	 Budapest, Hungary}
\email{szemered@cs.rutgers.edu}
\thanks{The fifth author is supported by ERC Advanced Grant 321104.}


\begin{abstract}
We show that every 3-uniform hypergraph with $n$ vertices and minimum vertex degree at least
$(5/9+o(1))\binom{n}2$ contains a tight Hamiltonian cycle. Known lower bound
constructions show that this degree condition is asymptotically optimal.
\end{abstract}

\maketitle

\section{Introduction}
G.~A.~Dirac~\cite{dirac} proved that every graph $G=(V,E)$ on at least $3$ vertices and with minimum vertex degree $\delta(G)\geq |V|/2$ contains a
Hamiltonian cycle. This result is best possible, as there are graphs~$G$ with minimum
degree $\delta(G)=\big\lceil|V|/2\big\rceil -1$ not containing a Hamiltonian cycle.

We continue the study to which extent Dirac's theorem can be generalised to hypergraphs.
Here we shall restrict to $3$-uniform hypergraphs and if not mentioned otherwise by a hypergraph we
will mean a 3-uniform hypergraph. Note that in this case there
are at least two natural concepts of a minimum degree condition and several
notions of cycle, and we briefly introduce some of them below.

For a hypergraph $H=(V,E)$  and a vertex $v\in V$ we denote by $N_H(v)$ \emph{the neighbourhood of $v$} and by
$d_H(v)$  \emph{the degree of $v$} defined as
\[
	N_H(v)=\{e\in E\colon v\in e\}\qqand d_H(v)=\big|N_H(v)\big|\,.
\]
Let $\delta(H)=\min d_H(v)$ be the \emph{minimum vertex degree} of $H$ taken over all $v\in V$.
Similarly,  for any two vertices $u$, $v\in V$ we denote by $N_H(u,v)$ their \emph{pair neighbourhood} and by~$d_H(u,v)$ their
\emph{pair degree} defined by
\[
	N_H(u,v)=\{e\in E\colon u, v\in e\}\qqand d_H(u,v)=\big|N_H(u,v)\big|\,.
\]
Let $\delta_2(H)=\min d_H(u,v)$ be the \emph{minimum pair degree} over
all pairs of vertices of~$H$. We will sometimes skip the subscript and write $d(v)$, $N(v)$, $d(u,v)$, and $N(u,v)$ instead.

An early notion of cycles in hypergraphs appeared in the
work of Berge~\cite{berge} (see also~\cite{Bermond}) more than 40 years ago.
More recently,  Katona and Kierstead~\cite{KK} considered
the following types of paths and cycles.
A hypergraph $P$ is a \emph{tight path of length $\l$},
if $|V(P)|=\l+2$ and there is an ordering of the vertices $V(P)=\{x_1,\dots,x_{\l+2}\}$
such that a triple~$e$ forms a hyperedge of~$P$  if and only if $e=\{x_i,x_{i+1},x_{i+2}\}$
for some $i\in[\l]$. The ordered pairs $(x_1,x_2)$ and $(x_{\l+1},x_{\l+2})$
are the \emph{end-pairs} of~$P$ and we say that $P$ is a tight $(x_1,x_2)$-$(x_{\l+1},x_{\l+2})$ path. This
definition of end-pairs is not symmetric and implicitly fixes a direction on~$P$ and the order of the end-pairs.
Hence, we may refer to $(x_1,x_2)$ as the \emph{starting pair} and to $(x_{\l+1},x_{\l+2})$ as the \emph{ending pair}.
All other vertices of $P$ are called \emph{internal}.
We sometimes identify such a path~$P$ with the sequence of its vertices $x_1\dots x_{\l+2}$.
Moreover, a \emph{tight cycle~$C$ of length~$\l\geq 4$} consists
of a path $x_1\dots x_\l$ of length $\l-2$ and the two additional hyperedges
$\{x_{\l-1},x_\l,x_1\}$ and $\{x_{\l},x_1,x_2\}$. In both cases the \emph{length} of a tight cycle
and of a tight path is measured by the number of hyperedges and we will use the same convention
for the length of cycles, paths, and walks in graphs. For simplicity we denote
edges and hyperedges by $xy$ and $xyz$ instead of~$\{x,y\}$ and $\{x,y,z\}$.

Roughly speaking, one may think of tight paths
and cycles as ordered hypergraphs such that ``consecutive'' edges overlap in exactly
two vertices. Similarly, one may consider so-called \emph{loose} paths and cycles, where the overlap
is restricted to one vertex only. Given a hypergraph $H$, a  cycle, tight or loose, is called \emph{Hamiltonian} if it is a subhypergraph of $H$ passing through all the vertices of $H$. The optimal approximate minimum pair and vertex degree conditions
for the existence of loose Hamiltonian cycles were obtained in~\cites{KO,BHS} and precise versions
for large hypergraphs appeared in~\cites{CM,YiJie}.
Results on pair degree conditions implying tight Hamiltonian cycles were obtained in~\cites{rrs3,3}.
For minimum vertex degrees, $(5/9-o(1))n^2/2$ provides a lower bound
(see Examples~\ref{exmpl1}\,\ref{it:exmpl:1}\,--\,\ref{it:exmpl:3} below), which was conjectured to be optimal.
So far only suboptimal upper bounds were obtained in~\cites{GPW,1112,RRSSz}.
We close this gap here, as the following result provides an
asymptotically
optimal minimum vertex degree condition for tight Hamiltonian cycles.
\begin{theorem}	\label{thm:main}
	For every $\alpha>0$ there exists an integer $n_0$ such that every
	$3$-uniform hypergraph $H$ with $n\ge n_0$ vertices and with minimum vertex
	degree $\delta(H)\geq \bigl(\tfrac59+\alpha\bigr)\tfrac{n^2}{2}$ contains
	a tight Hamiltonian cycle.
\end{theorem}
A recent result of Cooley and Mycroft~\cite{CoMy} establishes the existence of
an almost spanning tight cycle under the same degree condition as in Theorem~\ref{thm:main}.
Moreover, both these results
are asymptotically best possible, as the following well known examples show.

\begin{exmpl}\label{exmpl1}
\begin{enumerate}[label=\rmlabel]
	\item\label{it:exmpl:1} Consider a partition $X\dcup Y=V$ of a vertex set $V$ of size $n$
		with $|X|=\lceil (n+1)/3\rceil$ and let $H$ be the hypergraph containing all
		triples $e\in V^{(3)}$ such that $|e\cap X|\neq 2$. An averaging argument shows that
		a Hamiltonian cycle in $H$ would need to contain an edge $e$ with at least two vertices
		from $X$. Consequently~$e\subseteq X$ and the cycle could never ``leave'' $X$.
		Therefore, $H$ contains no Hamiltonian cycle (see, e.g.,~\cite{1112}).
		Moreover, we have  $\delta(H)\geq (5/9-o(1))n^2/2$.
	\item\label{it:exmpl:2} Similarly, one may consider a partition $X\dcup Y=V$ with $|X|=\lceil 2n/3\rceil$
		and let~$H$ be the hypergraph consisting of all
		triples $e\in V^{(3)}$ such that $|e\cap X|\neq 2$. Again $H$ has $\delta(H)\geq (5/9-o(1))n^2/2$
		and it contains no tight Hamiltonian cycle.
	\item\label{it:exmpl:3} The last example utilises the fact that every tight Hamiltonian cycle contains a matching of
		size~$\lfloor n/3\rfloor$. Again we consider a partition $X\dcup Y=V$ this time with $|X|=\lfloor n/3\rfloor-1$
		and let $H$ consist of all triples having at least one vertex in~$X$. Consequently,~$H$
		contains no matching of size $\lfloor n/3\rfloor$ and, hence, no tight Hamiltonian cycle.
		On the other hand, $\delta(H)\geq (5/9-o(1))n^2/2$.
\end{enumerate}
\end{exmpl}

We also would like to mention that in addition to the results on Hamiltonian cycles
in $3$-uniform hypergraphs discussed here, quite a few extensions and related results
for $k$-uniform hypergraphs already appeared in the literature and we refer to the surveys~\cites{sur,Zhao-sur}
(and the references therein) for a more detailed discussion.

\subsection*{Organisation} The proof of Theorem~\ref{thm:main} is based on the \emph{absorption method} developed in~\cite{rrs3}
and we discuss this approach in Section~\ref{sec:absmethod}. In Section~\ref{sec:main-proof}
we introduce the main concepts and lemmas for the proof of  Theorem~\ref{thm:main} and deduce
the theorem based on the lemmas. Each of the subsequent Sections~\ref{sec:robust}\,--\,\ref{sec:longpath}
is devoted to the proof of one of the main lemmas from Section~\ref{sec:main-proof}.

\section{Building Hamiltonian cycles in hypergraphs}
\label{sec:plan}
\subsection{Absorption method}
\label{sec:absmethod}
In~\cite{rrs3}
the  \emph{absorption method} was introduced, which
turned out to be a very well suited approach for extremal degree-type problems
forcing the existence of
spanning subhypergraphs. Our proof is also guided
by this strategy, which in the context of  Hamiltonian cycles can be summarised as
follows: Construct an almost spanning cycle~$C$ that contains a special,
so-called \emph{absorbing path}~$P_A$. The absorbing path has the special property
that it can \emph{absorb} the vertices outside~$C$ in such a way that
a Hamiltonian cycle is created.

For example, in the context of graphs a vertex $v$ outside~$C$ could be easily added to~$C$
if it formed a triangle with some edge $xy$ of~$C$, i.e., we would replace
the edge $xy$ of~$C$ by the path $x$-$v$-$y$ of length~$2$. Obviously, this
would have no effect on the remainder of~$C$, since $xy$ and the path
$x$-$v$-$y$ have the same end vertices. However, in order to repeat such a procedure
for $m$ vertices outside~$C$ it would be convenient if each such vertex would form a
triangle with at least $m$ mutually disjoint edges in~$P_A\subseteq C$. Then we could absorb one
vertex after another in a greedy manner into $P_A$ and its extensions.
However, in the proof we may not have much control on the
set of vertices left out by the almost spanning cycle~$C$.

In order to prepare for such a scenario we ensure that~$P_A$ can absorb
any set of vertices, which is not too large. For this it would be desirable to know
that for every vertex $v$ there exist many edges that form a triangle with~$v$, i.e.,
there are many \emph{$v$-absorbers}. Let us remark that if one would like to prove an
approximate version of Dirac's theorem for $n$-vertex graphs $G$
with $\delta(G)\geq (1/2+\alpha)n$,
then these edges would exist. Indeed, one can observe that the degree assumption forces for every vertex $v$
at least~$\alpha n^2/2$ triangles containing it.
Based on this fact one can show that $\eps n$  edges selected independently at random will
contain, with high probability, at least $\delta n$  $v$-absorbers for any vertex~$v$, for some suitably chosen
constants
satisfying $\alpha>\eps>\delta>0$. Moreover, the degree condition allows us to put
all these edges onto one  path,  an absorbing path~$P_A$ with the desired property. Consequently,
the problem of finding a Hamiltonian cycle reduces to finding an almost spanning cycle
$C$ containing~$P_A$ and covering all but at most $\delta n$ vertices of~$G$.

In the context of Dirac's theorem for graphs such a ``reduction'' seems to be somewhat
going overboard, as much simpler proofs even of the exact result are known.
However, for hypergraphs no such simple proof surfaced yet and the
absorption method seems to provide an appropriate approach.

For tight cycles in $3$-uniform hypergraphs, the following absorbers were
considered in~\cite{3}: two hyperedges $xyz$ and $yzw$ (which themselves form a tight
$(x,y)$-$(z,w)$-path of length~2) are a $v$-absorber if $v$ forms a hyperedge with each of the three
consecutive pairs $xy$, $yz$, and $zw$. These three hyperedges
 allow us to insert $v$ between $y$ and $z$, leading
to a tight path of length three with the same end-pairs $(x,y)$ and $(z,w)$.
It is not hard to show  that the minimum pair degree condition $\delta_2(H)\geq (1/2+\alpha)n$
for an $n$-vertex hypergraph $H$ guarantees the existence of $\Omega(n^4)$
$v$-absorbers for any vertex~$v$, which is a good starting point for
building an absorbing path in this context. However, for building such a path
(and  for creating an almost spanning tight cycle~$C$) we would also
need to connect the end-pairs of absorbers (and eventually
the end-pairs of paths to build up~$C$). Again, the minimum pair degree assumption
was utilised for these connections in~\cite{3} and it could be shown that any pair
of pairs can be connected by a relatively short tight path.

For the proof of Theorem~\ref{thm:main}, however, we only have a minimum vertex degree
condition at hand and this calls for more complex $v$-absorbers and a more complicated
connecting mechanism. In \cites{1112,RRSSz} this problem was addressed
by removing hyperedges containing pairs with too small degree, which led to suboptimal
minimum degree conditions. For the asymptotically optimal condition of $\big(\tfrac59+o(1)\big)\tfrac{n^2}2$
new ideas for the absorbers and the connectability were required.

Roughly speaking, the absorbers we shall use here consist of two parts. First, we show that
there are $\Omega(n)$ vertices~$z$ for which there exist $\Omega(n^4)$ tight paths $P_z=xyy'x'$, which can absorb~$z$ in the  way described above, and we call
such vertices~$z$~\emph{absorbable} (see Figure~\ref{fig:absorber}).
Moreover,
for every vertex~$v$ and every absorbable vertex~$z$
there are at least~$\Omega(n^4)$
quadruples $(a,b,c,d)$ such that both vertices $v$ and $z$ form a hyperedge with
all three pairs $ab$, $bc$, and $cd$. In particular, $abvcd$ and $abzcd$
form tight paths of length three in~$H$. Consequently, the two-edge path
$P_z=xyy'x'$ together with the three-edge path $abzcd$ can absorb~$v$ without
changing the end pairs of~$P_z$ and of $abzcd$. Indeed, we may
replace $z$ in $abzcd$ by  $v$ and then include $z$ between $y$ and $y'$ in $P_z$
(see Definition~\ref{dfn-vabs} and Figure~\ref{fig:absorber}). Most importantly, for every vertex $v$ such an argument
would give rise to~$\Omega(n^9)$  absorbers  consisting of a tight path
$abzcd$ of length three and a tight two-edge path~$P_z$, which, in principle,
would allow us to apply the absorption method in a similar manner as in~\cite{3}.

However, connecting the end-pairs of paths arising in the proof requires
more involved changes. In~\cite{3}, the minimum pair degree assumption
allows a \emph{Connecting Lemma} which asserts that for every pair of disjoint  pairs of vertices there exists
a relatively short tight path connecting them.

A similar statement in the context of Theorem~\ref{thm:main}
fails to be true.
In fact, there might be pairs of vertices
that are not contained in any hyperedge at all. More interestingly,
even when restricting to pairs of degree $\Omega(n)$, a corresponding connecting lemma might fail, as the following
example shows.

Similarly as in Examples~\ref{exmpl1}\,\ref{it:exmpl:1} and~\ref{it:exmpl:2}
consider a hypergraph $H=(V,E)$ with  partition $X\dcup Y=V$, where $|X|=\xi n$
for some $\xi<1/3$, and with an edge set defined by~$E=\{e\in V^{(3)}\colon |X\cap e|\neq 2\}$.
For sufficiently large $n$ such a hypergraph $H$ satisfies the degree condition in
Theorem~\ref{thm:main}, but every tight path~$P$ starting with a
pair  of vertices in~$X$ is bound to stay in $X$, i.e., $V(P)\subseteq X$.
Owing to such examples we will define a suitable notion of \emph{connectable pairs}, i.e.,
pairs of vertices for which a restricted Connecting Lemma can be proved (see Definition~\ref{def:connectable}
and Proposition~\ref{lem:con} in the next subsection).
On the other hand, this notion must be flexible and general enough, so that we can show that
all paths considered in the proof have such connectable pairs as ends. In fact,
this adjustment led to a few, somewhat technical, problems that we had to address here.
In the next section we present the notion of connectable pairs and the main lemmas which
lead to the proof of Theorem~\ref{thm:main}.

\subsection{Outline of the proof}
\label{sec:main-proof}
	In this section we present the proof of Theorem~\ref{thm:main} based on
	Propositions~\ref{prop:robust},~\ref{lem:con},~\ref{prop:reservoir},~\ref{prop:absorbingP}, and~\ref{prop:longpath}.
	These propositions will be stated here and we defer
	their proofs to separate later sections. The interplay of these propositions makes
	use of some auxiliary constants. For a simpler presentation we will note their
	dependencies along the
	way	by writing $a\gg b$ to indicate that $b$ will be chosen sufficiently small
	depending on~$a$ (and other constants appearing to the left of~$b$).
	
	More precisely,
	we are first given $\alpha>0$ by Theorem~\ref{thm:main} and without loss of generality we
	may assume that $1\gg\alpha$. Then we fix the following
	auxiliary constants $\beta$, $\zeta_*$, $\zeta_{**}$, $\theta_*$, $\theta_{**}>0$ and
	integers $\l$, $n\in \NN$ obeying the following hierarchy
	\begin{equation}\label{eq:consthierall}
		1
		\gg
		\alpha
		\gg
		\beta,\frac{1}{\l},\zeta_*
		\gg
		\theta_*\gg\zeta_{**}
		\gg
		\theta_{**}
		\gg
		\frac{1}{n}\,.
	\end{equation}
	These constants will be introduced together with the propositions  and
	the quantification of the propositions will allow us to fix them under the hierarchy given
	in~\eqref{eq:consthierall}.

	Theorem~\ref{thm:main} concerns $n$-vertex hypergraphs $H=(V,E)$ with minimum vertex degree
	$\delta(H)\geq \big(\frac{5}{9}+\alpha\big)\frac{n^2}{2}$. This degree condition
	implies a corresponding edge density of the \emph{link graphs} defined below.
	\begin{dfn} \label{dfn:link}
		For a $3$-uniform hypergraph $H=(V,E)$ and a vertex $v\in V$ we define
		the \emph{link graph $L_v$ of $v$} as the graph with vertex set
		$V(L_v)=V$ and edge set
		\[
			E(L_v)=\{ yz \colon vyz\in E(H) \}\,.
		\]
	\end{dfn}
	Observe that $v$ is an isolated vertex in the link graph $L_v$ and $e(L_v)=d_H(v)\geq \delta(H)$.
	The minimum degree assumption of Theorem~\ref{thm:main} implies that
	every link graph has density at least $5/9+\alpha$ and
	in Section~\ref{sec:robust}
	we investigate structural properties
	of such graphs. In particular,
	we shall show that these link graphs contain a ``well connected'' large subgraph,
	which will allow us to build and connect tight paths in the hypergraph (see Proposition~\ref{lem:con} below).
	More precisely, we consider subgraphs satisfying the following property.
	\begin{dfn} \label{dfn:robust}
		A graph $R$ is said to be \emph{$(\beta, \l)$-robust} if for any two distinct
		vertices~$x$ and~$y$ of~$R$ the number of $x$-$y$-paths in $R$ of length $\ell$ is at least
		$\beta |V(R)|^{\ell-1}$.
	\end{dfn}

	The following proposition, which will be proved in Section~\ref{sec:robust}, asserts that all
	link graphs contain a robust subgraph with many vertices and edges. For a graph $G$ and $A$,
	$B\subseteq V(G)$, let $e_G(A,B)$ be the number of edges of $G$ with one vertex in $A$ and one in $B$.

	\begin{prop}[Robust subgraphs] \label{prop:robust}
		For every $\alpha>0$ there are $\beta>0$ and an odd integer $\l\geq 3$
		such that for sufficiently large~$n$
		every $n$-vertex graph $L=(V,E)$ with $|E|\geq \left(\frac{5}{9}+\alpha\right)\frac{n^2}{2}$
		contains an induced subgraph $R\subseteq L$ satisfying
		\begin{enumerate}[label=\rmlabel]
		\item\label{it:rc1} $|V(R)|\ge \bigl(\frac23+\frac\alpha 2\bigr)n$,
		\item\label{it:rc2} $e_L\big(V(R),V\setminus V(R)\big)\leq \alpha n^2/4$ and
			$e(R)\geq \left(\frac{5}{9}+\frac{\alpha}{2}\right)\frac{n^2}{2}-\frac{(n-|V(R)|)^2}{2}$,
		\item\label{it:rc3} and $R$ is $(\beta,\l)$-robust.
	\end{enumerate}
	\end{prop}
	
	For the proof of Theorem~\ref{thm:main} we fix for every vertex $v\in V$ a
	$(\beta,\l)$-robust subgraph $R_v\subseteq L_v$
	as guaranteed by Proposition~\ref{prop:robust}. In other words, after $\alpha >0$ was revealed in Theorem~\ref{thm:main},
	we use
	Proposition~\ref{prop:robust} to define constants $\beta>0$ and $\l\in \NN$.  We indicate this dependency  by
	\[
		\alpha \gg \beta,\frac{1}{\l}\,.
	\]
	Moreover, we may assume that $n$ is sufficiently large, as it will
	be the last constant to be chosen in the proof of Theorem~\ref{thm:main}. Consequently,
	for any given hypergraph~$H=(V,E)$ concerned in
	Theorem~\ref{thm:main} we can appeal to Proposition~\ref{prop:robust} and this way we fix a
	$(\beta,\l)$-robust subgraph $R_v\subseteq L_v$ for every vertex $v\in V$. We summarise this in
	the following setup.
	
	\begin{setup}\label{setup}
		Suppose  $\alpha$, $\beta>0$, suppose $\l\geq 3$ is an odd integer,
		and suppose $H=(V,E)$ is a $3$-uniform hypergraph with $|V|=n$ sufficiently large,
		with $\delta(H)\geq \left(\frac{5}{9}+\alpha\right)\frac{n^2}{2}$,
		and with $(\beta,\l)$-robust subgraphs $R_v\subseteq L_v$ for every $v\in V$
		given by Proposition~\ref{prop:robust}.
	\end{setup}
	
	As discussed in Section \ref{sec:absmethod}, under the degree assumption of Theorem~\ref{thm:main}
	it is not necessarily true that any two pairs of vertices can be connected at all by a tight path,
	even if we only consider pairs of high degree. Still there is a
	reasonably large collection of pairs admitting such mutual connections. In fact, pairs that
	are contained in sufficiently many robust subgraphs can be connected by tight paths
	in~$H$. This will be made precise in the definition of \emph{connectable pairs}
	and in the \emph{Connecting Lemma} stated below.	
	\begin{dfn}\label{def:connectable}
		Given Setup~\ref{setup} and $\zeta>0$, an \emph{unordered} pair $xy$ of vertices in $V$ is said to be
		\emph{$\zeta$-connectable} if the set
		\[
			U_{xy}=\{v\in V \colon xy\in E(R_v)\}
		\]
		of all vertices $v$ having $xy$ as an edge of their robust subgraph has size
		$|U_{xy}|\ge\zeta |V|$.
		The \emph{ordered} pair $(x, y)$ is called \emph{$\zeta$-connectable} if $xy$ is.
	\end{dfn}
	
	The Connecting Lemma below asserts that pairs of connectable pairs can be connected by many tight paths.
	Section~\ref{sec:connecting} is devoted to the proof of Proposition~\ref{lem:con}.
	
	\begin{prop}[Connecting Lemma] \label{lem:con}
		Given Setup~\ref{setup} and $\zeta>0$, there exists  $\theta>0$ such that
		every two disjoint $\zeta$-connectable ordered pairs~$(x, y)$ and $(z, w)$
		are connected by at least $\theta n^{3\ell+1}$ tight $(x, y)$-$(z, w)$-paths of
		length $3(\ell+1)$ in~$H$.
	\end{prop}
	
	The Connecting Lemma  plays a crucial role in building an absorbing path $P_A$
	(guaranteed by Proposition~\ref{prop:absorbingP}), as well as in building an almost spanning
	cycle $C$ (see Proposition~\ref{prop:longpath} below). For the former application
 we shall fix $\zeta_*$ with $\alpha\gg \zeta_*$ and  the  Connecting Lemma
	will yield some constant $\theta_*$ with $\zeta_*\gg \theta_*$.
	Given $\theta_*$ we will then choose $\zeta_{**}$ for the latter application,
 obtaining $\theta_{**}$ with $\zeta_{**}\gg \theta_{**}$. This gives rise to the hierarchy
	\[
		\alpha\gg \zeta_*\gg\theta_*\gg\zeta_{**}\gg\theta_{**}\,,
	\]
	as declared in~\eqref{eq:consthierall}.
	
	The Connecting Lemma will allow us to connect tight paths that start and end
	with a connectable pair. However, in the process of building longer paths, we
	must not interfere with already constructed subpaths. For that we put a small \emph{reservoir}
	of vertices aside and in the proof of Proposition~\ref{prop:longpath}
	connections will only be created by using new vertices from this reservoir.
	The existence of such a reservoir set is given by the following proposition and its
	probabilistic proof is given in Section~\ref{sec:reservoir}.
	
	\begin{prop}[Reservoir Lemma]\label{prop:reservoir}
		Given Setup~\ref{setup} and,  in addition, let $\theta_*$, $\zeta_{**}>0$ and suppose that
		$\theta_{**}=\theta_{**}(\alpha,\beta,\l,\zeta_{**})>0$ is given by Proposition~\ref{lem:con}.
	
		There exists a reservoir set $\cR\subseteq V$ with
		$\frac{\theta_*^2}{2}n\leq|\cR|\leq \theta_*^2n$
		such that for all disjoint pairs of $\zeta_{**}$-connectable pairs~$(x,y)$ and~$(z,w)$
		there are at least $\theta_{**} |\cR|^{3\ell+1}/2$ tight
		$(x, y)$-$(z, w)$-paths of length $3(\ell+1)$ in $H$ whose internal vertices
		belong to $\cR$.
	\end{prop}
	We summarise the situation by the following setup extending Setup~\ref{setup}.
	
	\begin{setup}\label{setup2}
		Let Setup~\ref{setup} and constants as stated in~\eqref{eq:consthierall}
		be given, where $\theta_{*}\!=\!\theta_{*}(\alpha,\beta,\l,\zeta_{*})$
		and $\theta_{**}=\theta_{**}(\alpha,\beta,\l,\zeta_{**})$ are given
		by Proposition~\ref{lem:con}. In addition, let $\cR\subseteq V$  be a reservoir set
		given by Proposition~\ref{prop:reservoir}.
	\end{setup}

	After these preparatory propositions we are ready to build a Hamiltonian cycle.
	As outlined above, we first create and put aside an \emph{absorbing path~$P_A$}, which at the end of the proof
	will allow us to `absorb' an arbitrary (but not too large) set $X$ of leftover vertices into an almost
	spanning tight cycle, thus creating a tight Hamiltonian cycle.
	\begin{prop}[Absorbing path]\label{prop:absorbingP}
		Given Setup~\ref{setup2}, there exists a tight (absorbing) path~$P_A$ which is a subhypergraph
		of $H-\cR$ and has the following properties.
		\begin{enumerate}[label=\rmlabel]
		\item\label{it:PA1} $ |V(P_A)|\le \theta_*n$,
		\item\label{it:PA2} the end-pairs of $P_A$ are $\zeta_*$-connectable, and
		\item\label{it:PA3} for every set $X\subseteq V\setminus V(P_A)$
			with $|X|\le 2\theta_*^2n$ there is a tight path in $H$ whose set of vertices
			is $V(P_A)\cup X$ and whose end-pairs are the same as those of $P_A$.
		\end{enumerate}
	\end{prop}
	The proof of Proposition~\ref{prop:absorbingP} is the content of Section~\ref{sec:absorbing}.
	The last proposition (see Section~\ref{sec:longpath} for its proof)
	establishes the existence of an  almost spanning tight cycle
	containing~$P_A$ and covering all but at most $2\theta_*^2n$ vertices of $H$.
	
	\begin{prop}[Almost spanning cycle]\label{prop:longpath}
		Given Setup~\ref{setup2} and a tight absorbing path $P_A\subseteq H-\cR$
		from Proposition~\ref{prop:absorbingP},
		there exists a tight cycle $C\subseteq H$ containing $P_A$ and passing through at least
		$(1-2\theta_{*}^2)n$ vertices.
	\end{prop}

	Finally, we observe that combining Propositions~\ref{prop:absorbingP} and~\ref{prop:longpath} implies
	the existence of a Hamiltonian tight cycle in~$H$.
	\begin{proof}[Proof of Theorem~\ref{thm:main}]
		Given $\alpha>0$ we choose all auxiliary constants as described above
		and assume Setup~\ref{setup2}. Proposition~\ref{prop:absorbingP} yields
		an absorbing path $P_A$ and then Proposition~\ref{prop:longpath}
		guarantees the existence of an almost spanning cycle~$C$ which contains the
		absorbing path $P_A$ and covers all but at most $2\theta_*^2n$ vertices.
		Property~\ref{it:PA3} of the absorbing path $P_A$ allows us to absorb the
		remaining vertices into the cycle. This concludes the proof of Theorem~\ref{thm:main}. 	
	\end{proof}
	
	It is left to prove
	Propositions~\ref{prop:robust},~\ref{lem:con},~\ref{prop:reservoir},~\ref{prop:absorbingP},
	and~\ref{prop:longpath}, which is the content of Sections~\ref{sec:robust}\,--\,\ref{sec:longpath}.

\section{Robust subgraphs} \label{sec:robust}
In this section we establish the existence of robust subgraphs within the link graphs of the given
hypergraph~$H$.
The proof of Proposition~\ref{prop:robust} splits into two parts. In the first part
(rendered in Lemma~\ref{lem:robust1} below)
we establish the existence of a subgraph~$R$ satisfying properties~\ref{it:rc1}
and~\ref{it:rc2} of Proposition~\ref{prop:robust}, and the following strong connectivity property.

\begin{dfn} \label{def:nonsep}
	A graph $R$ is said to be \emph{$\mu$-inseparable} if $\delta(R)\geq \mu |V(R)|$
	and for every partition $X\dcup Y=V(R)$ into parts of size at least $\mu|V(R)|$ we have
	$e(X,Y)\geq \mu^2|V(R)|^2$.
\end{dfn}
\begin{lemma}\label{lem:robust1}
For every $\alpha>0$ and sufficiently large~$n$
every $n$-vertex graph $L=(V,E)$ with $|E|\geq \left(\frac{5}{9}+\alpha\right)\frac{n^2}{2}$
contains an induced subgraph $R\subseteq L$ satisfying
\begin{enumerate}[label=\rmlabel]
\item\label{it:rc1p} $|V(R)|\ge \bigl(\frac23+\frac\alpha 2\bigr)n$,
\item\label{it:rc2p} $e_L\big(V(R),V\setminus V(R)\big)\leq \alpha n^2/4$ and
	$e(R)\geq \left(\frac{5}{9}+\frac{\alpha}{2}\right)\frac{n^2}{2}-\frac{(n-|V(R)|)^2}{2}$,
\item\label{it:rc3p} and $R$ is $(\alpha/72)$-inseparable.
\end{enumerate}
\end{lemma}

In the second part of the proof we deduce
Proposition~\ref{prop:robust} from Lemma~\ref{lem:robust1}
and for that we utilise the inseparability
of~$R$ to deduce the robustness. We first give the proof of the lemma.

\begin{proof}[Proof of Lemma~\ref{lem:robust1}]
	We may assume $\alpha\in(0,4/9]$, since otherwise no graph $L$ satisfying the assumption exists.
	For convenience set
	\begin{equation}\label{eq:rc:mu}
		\mu=\frac{\alpha}{72}
	\end{equation}
	and for sufficiently large $n$ let $L=(V,E)$ be an $n$-vertex  graph with
	$e(L)\geq \big(\frac{5}{9}+\alpha\big)\frac{n^2}{2}$.
	
	\subsection*{Defining the subgraph \texorpdfstring{$R$}{{\it R}}} We fix the maximum $t\in\NN$ for which there exists a
	partition $V_1\dcup\dots\dcup V_t=V$ with
	\begin{enumerate}[label=\alabel]
	\item\label{it:rc:p:1} $|V_1|\geq \dots\geq |V_t|\geq \mu n/2$ and
	\item\label{it:rc:p:2} $\sum_{1\leq i<j\leq t}e_{L}(V_i,V_j)\leq 2(t-1)\mu^2 n^2$.
	\end{enumerate}
	Since the trivial partition $V_1=V$ satisfies properties~\ref{it:rc:p:1} and~\ref{it:rc:p:2}
	we know $t\geq 1$ and from~\ref{it:rc:p:1} we infer that $t\leq 2/\mu$. Moreover, the upper bound
	on~$t$ combined with~\ref{it:rc:p:2} implies
	\begin{equation}\label{eq:b'}
		\sum_{1\leq i<j\leq t}e_{L}(V_i,V_j)
		<
		4\mu n^2\,.
	\end{equation}
	Let $\eta\in (0,1]$ be given by
	\[
		|V_1|=\eta n\,.
	\]
	It is easy to check that $\eta>1/3$, as otherwise
	\begin{align*}
		e(L)
		=
		\sum_{i=1}^te_L(V_i)+\sum_{1\leq i<j\leq t}e_{L}(V_i,V_j)
		& <
		\sum_{i=1}^t\frac{|V_i|^2}{2}+4\mu n^2 \\
		& \leq
		\frac{n}{3}\sum_{i=1}^t\frac{|V_i|}{2}+4\mu n^2
		=
		\left(\frac{1}{3}+8\mu\right)\frac{n^2}{2}
		\overset{\eqref{eq:rc:mu}}{\leq}\frac{5}{9}\frac{n^2}{2}
	\end{align*}
	contradicts our assumption on $e(L)$.
	However, below we even show $\eta>2/3$ and in
	the proof of that we will consider a quadratic inequality
	where the weak bound $\eta>1/3$ from above rules out one interval
	of possible solutions.
	In fact, we have
	\[
		\frac{\eta^2n^2}{2}
		\geq
		e_L(V_1)
		>
		e(L)-\frac{(n-|V_1|)^2}{2}-4\mu n^2
		\geq
		\left(\frac{5}{9}+\alpha-(1-\eta)^2-8\mu\right)\frac{n^2}{2}\,.
	\]
	This leads to the quadratic inequality
	\[
		\eta^2\geq \frac{5}{9}+\alpha - (1 -\eta)^2 -8\mu
		\quad\Longleftrightarrow\quad
		\left(\eta-\frac{1}{3}\right)\left(\eta-\frac{2}{3}\right)\geq \frac{\alpha}{2}-4\mu\,.
	\]
	Since assuming that $\eta\in(\tfrac{1}{3}, \tfrac{2}{3}+\tfrac{2}{3}\alpha)$ would yield
	\[
		\left(\eta-\frac{1}{3}\right)\left(\eta-\frac{2}{3}\right)
		<
		\left(\eta-\frac{1}{3}\right)\cdot\frac{2}{3}\alpha
		<
		\frac{2}{3}\cdot\frac{2}{3}\alpha
		=
		\frac{\alpha}{2}-\frac{\alpha}{18}
		=
		\frac{\alpha}{2}-4\mu\,,
	\]
	we have
	\begin{equation}\label{eq:rc:V1}
		|V_1|
		=
		\eta n
		\geq
		\left(\frac{2}{3}+\frac{2}{3}\alpha\right)n
		\qqand
		e_L(V_1)>\frac{2}{9}n^2\geq \mu n^2\,.
	\end{equation}
	Let $W=\{w_1,\dots,w_m\}\subseteq V_1$ be a maximal (ordered) subset such that
	\[
		\big|N_L(w_i)\cap (V_1\setminus \{w_1,\dots,w_{i-1}\})\big|
		<\mu n
	\]
	for every $i\in [m]$. Owing to the second part of~\eqref{eq:rc:V1} we have
	$V_1\setminus W\neq\emptyset$. Moreover, by definition $V_1\setminus W$ induces a subgraph
	of minimum degree at least $\mu n$ in $L$
	and we set
	\[
		U=V_1\setminus W
		\qqand
		R=L[U]\,,
	\]
	and below we verify that~$R$ has the desired properties.
	
	\subsection*{Verifying the properties of \texorpdfstring{$R$}{{\it R}}}
	We first observe that $|W|<\mu n/2$. Suppose for a contradiction
	that there exists a subset $W'=\{w_1,\dots,w_{\lceil \mu n/2\rceil}\}\subseteq W$. Then we can
	replace the set $V_1$
	in the partition
	$V_1\dcup\dots\dcup V_t=V$ by $W' \dcup (V_1\setminus W')$ and obtain a partition into $t+1$ parts,
	which satisfies~\ref{it:rc:p:1}, as
	\begin{equation}\label{eq:deltaR}
		|V_1\setminus W'|\geq |V_1\setminus W|=|U|>\delta(R)\geq \mu n.
	\end{equation}
	Moreover, the ordering of the vertices in~$W$ yields
	\begin{equation}\label{eq:rc:eV1U}
			e_L(W', V_1\setminus W')
			\leq
			\sum_{w_i\in W'}\big|N_L(w_i)\cap (V_1\setminus \{w_1,\dots,w_{i-1}\})\big|
			<
			\mu n\cdot |W'|
			\leq
			\mu^2 n^2,
	\end{equation}
	which shows that the partition $W'\dcup (V_1\setminus W')\dcup V_2\dcup\dots\dcup V_t=V$
	also satisfies~\ref{it:rc:p:2}. Consequently, this partition would contradict the
	maximal choice of~$t$ and, hence, we have indeed $|W|<\mu n/2$.
	
	Property~\ref{it:rc1p} of Lemma~\ref{lem:robust1} then follows from
	\begin{align*}
		|V(R)|=|U|=|V_1\setminus W|&=|V_1|-|W|\\
		&>
		|V_1|-\frac{\mu}{2} n
		=
		(\eta-\tfrac{\mu}{2})n
		\overset{\eqref{eq:rc:V1}}{\geq}
		\left(\frac{2}{3}+\frac{2\alpha}{3}-\frac{\mu}{2}\right)n
		\overset{\eqref{eq:rc:mu}}{\geq}
		\left(\frac{2}{3}+\frac{\alpha}{2}\right)n.	
	\end{align*}
	For property~\ref{it:rc2p}, note that
	\begin{align*}
		e_L\big(U,V\setminus U\big)
		&=
		\sum_{i=2}^te_{L}(U,V_i)+e_L(U,W)\\
		&\leq
		\sum_{i=2}^te_{L}(V_1,V_i)+\mu n|W|
		\overset{\text{\ref{it:rc:p:2}}}{\leq}
		2(t-1)\mu^2n^2+\mu^2 n^2
		<
		4\mu n^2,
	\end{align*}
	where we used $t\leq 2/\mu$ in the last inequality. Consequently,
	the first inequality of property~\ref{it:rc2p} follows from the choice of $\mu$ in~\eqref{eq:rc:mu}.
	The second inequality is a direct consequence of the first and the lower bound on $e(L)$ given by
	the assumption of the lemma
	\[
		e(R)
		=
		e(L)-e_L(U,V\setminus U)-e_L(V\setminus U)
		\geq
		\left(\frac{5}{9}+\alpha\right)\frac{n^2}{2}-\frac{\alpha}{2}\frac{n^2}{2}-\frac{(n-|U|)^2}{2}\,.
	\]

	For property~\ref{it:rc3p}  we first note that we already observed the
	required minimum degree condition $\delta(R)\geq \mu |U|$
	in~\eqref{eq:deltaR}.
	For the second property in Definition~\ref{def:nonsep} consider an arbitrary partition
	$X\dcup Y=U$ with parts of size at least~$\mu |U|>2\mu n/3$. We appeal to the maximality of~$t$ and infer
	from~\ref{it:rc:p:2} that
	\[
		e_L(X,V_1\setminus X)>2\mu^2n^2\,.
	\]
	Consequently, since $V_1\setminus X=Y\dcup W$, we have
	\[
		e_R(X,Y)
		=
		e_L(X,Y\cup W)-e_L(X,W)
		\geq
		e_L(X,V_1\setminus X) - e_L(U,W)
		\geq
		2\mu^2n^2-\mu^2n^2=\mu^2n^2,
	\]
	which implies that $R$ is $\mu$-inseparable and this concludes the proof of
	Lemma~\ref{lem:robust1}. 	
\end{proof}

Next we deduce Proposition~\ref{prop:robust} from Lemma~\ref{lem:robust1}.

\begin{proof}[Proof of Proposition~\ref{prop:robust}]		
	For $\alpha\in(0,4/9]$ set $\mu=\alpha/72$.
	We set $\l$ to be the smallest odd integer such that
	\begin{equation}\label{eq:rc:betal}
		\l>\frac{8}{\mu^2}+1
		\qquad\text{and set}\qquad
		\beta=\frac{1}{72\ell}\left(\frac{\mu}{2}\right)^{6\l}\,.
	\end{equation}
	For sufficiently large $n$ let $L=(V,E)$ be an $n$-vertex graph with
	$e(L)\geq \big(\frac{5}{9}+\alpha\big)\frac{n^2}{2}$. Moreover, let
	$U\subseteq V$ and $R=L[U]$ be the induced subgraph guaranteed by
	Lemma~\ref{lem:robust1}. In particular, 	$V(R)=U$,
	\begin{equation}\label{eq:rc:U}
		|U|\geq \left(\frac{2}{3}+\frac{\alpha}{2}\right)n\,,\quad
		e(R)\geq \left(\frac{5}{9}+\frac{\alpha}{2}\right)\frac{n^2}{2}-\frac{(n-|U|)^2}{2}\,,
		\qand
		\delta(R)\geq \mu |U|\,.
	\end{equation}

	It remains to show that $R$ is $(\beta,\l)$-robust for the choice of
	$\beta$ and $\l$ in~\eqref{eq:rc:betal}. This proof will be carried out in
	three steps. First we show that for every pair of distinct vertices
	$x$, $z\in V(R)$ there exist at least $\Omega(n^{s-1})$ $x$-$z$-walks
	in $R$ of length~$s=s(x,z)\leq \l$ (see~\eqref{eq:rc:step1} below).
	In the second step we ensure that
	$s(x,z)$ can be chosen to be odd (see~\eqref{eq:rc:step2})
	and in the last step we show that we can insist
	that the walks have length $\l$ independent of the pair $x$ and $z$.
	Noting that most of these walks will indeed be paths then concludes
	the proof.
	Below we give the details of each
	of the three steps.
	
	\subsection*{First step} For an arbitrary vertex $x\in U$ and for every integer $i\geq 1$ we define
	\[
		Y^i_x=\{y\in U\colon \text{there are at least $(\mu^4/4)^s|U|^{s-1}$ $x$-$y$-walks of length $s$
			in $R$ for some $s\leq i$}\}\,.
	\]
	For every $i\geq 2$ we have $Y^i_x\supseteq Y^{i-1}_x$ and, consequently,
	\[
		|Y^i_x|\geq |Y^1_x|\geq|N_R(x)|\geq \delta(R)\geq \mu |U|\,.
	\]
	Next we show that for every integer $i$ with $1\leq i\leq 2/\mu^2$ at least one of the following  holds:
	\begin{equation}\label{eq:rc:claim}
		\big|U\setminus Y^i_x\big|<\mu |U|
		\qquad\text{or}\qquad
		\big|Y^{i+1}_x\setminus Y^i_x\big|\geq \frac{\mu^2}{2}|U|\,.
	\end{equation}
	
	If $|U\setminus Y^i_x|\geq \mu |U|$, then the $\mu$-inseparability of~$R$
	implies
	\begin{equation*}\label{eq:rc:claim:1}
		e_L\big(Y^i_x,U\setminus Y^i_x\big)\geq \mu^2|U|^2\,.
	\end{equation*}
	This means however that at least $\mu^2 |U|/2$ vertices $U\setminus Y^i_x$
	have at least $\mu^2|U|/2$ neighbours in $Y^i_x$. For every such vertex
	in $U\setminus Y^i_x$ at least $1/i\geq \mu^2/2$ proportion of its neighbours
	in~$Y^i_x$ are connected to~$x$ by walks of the same length, which
	implies $\big|Y^{i+1}_x\setminus Y^i_x\big|\geq \mu^2|U|/2$
	and this establishes~\eqref{eq:rc:claim}.
	
	From~\eqref{eq:rc:claim} we infer that for $j=\lfloor 2/\mu^2\rfloor$ we have
	$|U\setminus Y^j_x|<\mu |U|$. Since $x\in U$ was arbitrary, the same
	conclusion holds for every vertex $z\in U$, i.e., we also have
	$|U\setminus Y^j_z|<\mu |U|$. Therefore, at least $|U|-2\mu |U|> |U|/2$ vertices~$y$
	are contained in the intersection $Y^j_x\cap Y^j_z$. Each of these vertices
	gives rise to constants $s_1$, $s_2\leq j\leq 2/\mu^2$ such that there are at least
	$(\mu^4/4)^{s_1}|U|^{s_1-1}$ $x$-$y$-walks of length $s_1$
	and there are at least $(\mu^4/4)^{s_2}|U|^{s_2-1}$ $z$-$y$-walks of length $s_2$.
	Consequently, for $s_y=s_1+s_2\geq 2$ there are at least $(\mu^4/4)^{s_y}|U|^{s_y-2}$
	$x$-$z$-walks of length $s_y$ in $R$ passing through~$y$. Repeating
	this argument for all vertices $y\in Y^j_x\cap Y^j_z$ shows that there is a subset of at least
	$\frac{|U|}{2}/\frac{4}{\mu^4}$ vertices yielding the same pair $(s_1, s_2)$ and, hence, the
	same value $s_y$. Consequently, for some
	 $s(x,z)$ with $2\leq s(x,z)\leq 4/\mu^2$ there are at least
	\begin{equation}\label{eq:rc:step1}
		\frac{\mu^4}{8}|U|\cdot \left(\frac{\mu^4}{4}\right)^{s(x,z)} |U|^{s(x,z)-2}
		\geq
		\left(\frac{\mu}{2}\right)^{6s(x,z)}|U|^{s(x,z)-1}
	\end{equation}
	$x$-$z$-walks of length $s(x,z)$ in~$R$.
	It will be convenient to define for every pair of vertices $x$, $z\in U$ the set
	\begin{equation}\label{eq:rc:step1a}
		\cS_{x,z}=\big\{s\geq 2\colon \text{there are at least $(\mu/2)^{6s}|U|^{s-1}$ $x$-$z$-walks in $R$}\big\}
	\end{equation}
	and~\eqref{eq:rc:step1} asserts $\cS_{x,z}\cap [2,4/\mu^2]\neq \emptyset$.
	 This concludes the discussion of the first step.
	
	\subsection*{Second step} We elaborate on~\eqref{eq:rc:step1} and show that we can obtain
	a similar formula with the additional restriction that $s(x,z)$ is odd.	 For that let
	$x\in U$ be arbitrary and consider the disjoint sets
	\[
		Y^{\odd}_x\dcup Y^{\even}_x\subseteq U
	\]
	defined through the parity of the integers $s(x,y)$ for which the lower bound in~\eqref{eq:rc:step1} holds for
	the number of $x$-$y$-walks in $R$, i.e.,
	\begin{align*}
		Y^{\odd}_x&=\big\{y\in U\colon
			\text{$\cS_{x,y}\cap [2,4/\mu^2]$ contains only odd integers} \big\}
	\intertext{and}
		Y^{\even}_x&=\big\{y\in U\colon
			\text{$\cS_{x,y}\cap [2,4/\mu^2]$ contains only even integers} \big\}\,.
	\end{align*}
	Moreover, we consider the set $Y^{\flex}_x$ of ``parity-wise flexible'' vertices covering
	the remainder of $U$, i.e.,
	\[
		Y^{\flex}_x
		=
		\big\{y\in U\colon
			\text{$\cS_{x,y}\cap [2,4/\mu^2+1]$ contains both odd and even integers} \big\}\,.
	\]
	Owing to the additional ``$+1$'' in the definition, the set $Y^{\flex}_x$ may not be disjoint
	from $Y^{\odd}_x\cup Y^{\even}_x$. However, all three sets together cover $U$.
	More importantly,
	the vertices $y\in Y^{\flex}_x$ connect to $x$ by many odd and many even walks of short length,
	which will allow us to ``fix'' the parity for every vertex $z\in U$ by first connecting $z$ with
	some $y\in Y^{\flex}_x$
	and then, depending on the parity of the $z$-$y$-walk, continuing by a walk of different parity to~$x$.
	Obviously, for such an approach it will be useful that $Y^{\flex}_x$ indeed contains many vertices
	and, therefore, below we show
	\begin{equation}\label{eq:rc:step2a}
		\big|Y^{\flex}_x\big|\geq \frac{n}{36}\geq \frac{|U|}{36}\,.
	\end{equation}
	For that we note that $Y^{\odd}_x\setminus Y^{\flex}_x$ induces at most $\mu |U|^2$ edges,
	as otherwise some vertex in $y\in Y^{\odd}_x\setminus Y^{\flex}_x$ would have
	at least $2\mu |U|$ neighbours in $Y^{\odd}_x$. Any such a neighbour $y'$ and its odd $x$-$y'$-walks
	can be used to build even $x$-$y$-walks of length at most $4/\mu^2+1$ and at least a $(2/\mu^2+1)^{-1}$ proportion of
	these walks would have the same length. Consequently, there would be some even integer contained in
	$\cS_{x,y}\cap[2, 4/\mu^2+1]$, which contradicts
	$y\in Y^{\odd}_x\setminus Y^{\flex}_x$. Applying the same argument to $Y^{\even}_x\setminus Y^{\flex}_x$
	tells us
	\[
		e_R\big(Y^{\odd}_x\setminus Y^{\flex}_x\big) + e_R\big(Y^{\even}_x\setminus Y^{\flex}_x\big)	
		\leq
		2\mu |U|^2\,.
	\]	
	Since, trivially, $e_R\big(Y^{\odd}_x\setminus Y^{\flex}_x,Y^{\even}_x\setminus Y^{\flex}_x\big)\leq |U|^2/4$
	and all edges of~$R$ not counted
	so far are incident with a vertex in $Y^{\flex}_x$,
	we have
	\begin{align*}
		e(R)
		&\overset{\phantom{\eqref{eq:rc:U}}}{\leq}
		\left(\frac{1}{2}+4\mu\right)\frac{|U|^2}{2}+\sum_{v\in Y^{\flex}_x} d_R(v)\,.
	\intertext{On the other hand, we have}
		e(R)
		&\overset{\eqref{eq:rc:U}}{\geq}
		\left(\frac{5}{9}+\frac{\alpha}{2}\right)\frac{n^2}{2}-\frac{(n-|U|)^2}{2}
		\,.
	\end{align*}
	For $\rho$ defined by $|U|=\rho n$ these two estimates on $e(R)$ lead to
	\[
		\frac{2}{n}|Y^{\flex}_x|
		\geq
		\frac{2}{n}\sum_{v\in Y^{\flex}_x}\frac{d_R(v)}{n}
		\geq
		\left(\frac{5}{9}+\frac{\alpha}{2}\right)-(1-\rho)^2-\left(\frac{1}{2}+4\mu\right)\rho^2
		\geq
		\rho\left(2-\frac{3}{2}\rho\right)-\frac{4}{9}\,,
	\]
	where we used the choice $\mu=\alpha/72<\alpha/8$ for the last inequality.
	Since $\rho\in(2/3,1]$, the right-hand side is minimised for $\rho=1$ and~\eqref{eq:rc:step2a}
	follows.
	
	Having established~\eqref{eq:rc:step2a} below we shall show that for every vertex $z\in U$
	there exists some odd integer $s'(x,z)\leq 8/\mu^2+1$ such that there are at least
	\begin{equation}\label{eq:rc:step2}
			\frac{1}{36}\left(\frac{\mu}{2}\right)^{6s'(x,z)+4}|U|^{s'(x,z)-1}
	\end{equation}
	$x$-$z$-walks of length $s'(x,z)$ in~$R$. In fact, for every vertex~$z$ and every
	$y\in Y^{\flex}_x$ we appeal to~\eqref{eq:rc:step1a} and obtain many
	$z$-$y$-walks of length $s(z,y)$. Since $y\in Y^{\flex}_x$,  there is some
	\[
		s(y,x)\in \cS_{y,x}\cap [2,4/\mu^2+1]
	\]
	of different parity than $s(z,y)$ and connecting
	the corresponding walks gives us
	\[
		\left(\frac{\mu}{2}\right)^{6s(z,y)}|U|^{s(z,y)-1}
		\times
		\left(\frac{\mu}{2}\right)^{6s(y,x)}|U|^{s(y,x)-1}
		=
		\left(\frac{\mu}{2}\right)^{6s(z,y)+6s(y,x)}|U|^{s(z,y)+s(y,x)-2}
	\]
	$x$-$z$-walks of odd length $s(z,y)+s(y,x)\leq 8/\mu^2+1$ passing through~$y$.
	Similarly as in the first step
	we repeat
	this argument for all vertices $y\in Y^{\flex}_x$ and conclude that there must be a subset of
	$\frac{|U|}{36}/\frac{8}{\mu^4}$ vertices $y$ leading to the same pair $\bigl(s(z, y), s(y, x)\bigr)$
	with odd sum and thus to odd walks of the same length $s'(x,z)$.
	Hence, there are at least
	\[
		\frac{|U|}{36}\cdot\frac{\mu^4}{8}\cdot\left(\frac{\mu}{2}\right)^{6s'(x,z)}|U|^{s'(x,z)-2}
	\]
	$x$-$z$-walks of length $s'(x,z)$ in $R$ and~\eqref{eq:rc:step2} follows.
	
	\subsection*{Third step} In the last step we finally show that $R$ is $(\beta,\l)$-robust.
	So far we achieved in the second step that for every pair of vertices there are many
	short walks of odd length connecting them. However, so far the length may depend on the pair
	that is connected and below we extend many walks so that they all have the same length $\l$
	independent of the pair.
	In fact,
	we shall show that for every pair of distinct vertices $x$ and $z$ in $R$ there are at least
	$2\beta |U|^{\l-1}$ $x$-$z$-walks of length $\l$ in $R$.
	
	For an arbitrary vertex $x\in U$ we consider its neighbourhood $N_R(x)$ and
	let $S_R(x)$ be its  second neighbourhood, i.e.,  the set of vertices
	connected by a walk of length two with~$x$ in~$R$. In particular, $N_R(x)$ and
	$S_R(x)$ might not be disjoint. Since $\delta(R)\geq \mu |U|$, we have
	\begin{equation}\label{eq:rc:NSx}
		\big|N_R(x)\big|\geq \mu |U|
		\qqand
		e_R\big(N_R(x),S_R(x)\big)\geq \frac{1}{2}\mu |U|\cdot\big|N_R(x)\big|\geq \frac{\mu^2}{2} |U|^2\,,
	\end{equation}
	where the factor $1/2$ takes into account that $N_R(x)$ and $S_R(x)$ may not be disjoint.
	Consequently one can show that there are subsets $N_x\subseteq N_R(x)$ and  $S_x\subseteq S_R(x)$
	of size at least $\mu^2|U|/4$
	such that for every vertex $y\in N_x$ we have
	\[
		\big|N_R(y)\cap S_x\big|\geq \frac{\mu^2}{4} |U|
	\]
	and, similarly, $\big|N_R(y')\cap N_x\big|\geq \mu^2 |U|/4$ for every $y'\in S_x$.
	Indeed the sets $N_x$ and $S_x$ exist, as otherwise we could keep deleting edges incident
	to vertices of small degree in $N_R(x)$ (resp.\ $S_R(x)$). More precisely, we consider vertices
	one by one and if $v\in N_x$ (resp.\ $S_x$) has fewer than $\mu^2 |U|/4$ neighbours in $S_x$ (resp.\ $N_x$),
	then we remove the edges between $v$ and its neighbourhood in~$S_x$ (resp.\ $N_x$).
	However, this way less than
	\[
		 \big(\big|N_R(x)\big|+\big|S_R(x)\big|\big)\cdot \frac{\mu^2}{4}|U|
		\leq
		\frac{\mu^2}{2}|U|^2
	\]
	edges would be deleted altogether, which
	by~\eqref{eq:rc:NSx} implies that the procedure ends with a non-empty subgraph with the required degree
	condition.
	Therefore, for every vertex $y'\in S_x$ and every odd integer $s''$
	there exist at least $(\mu^2/4)^{s''}|U|^{s''}$ walks of length $s''$ that start in $y'$ and
	end in $N_x\subseteq N_R(x)$.
	
	Let $z\in U$ be distinct from~$x$. For every $y'\in  S_x$
	there is an odd integer $s'(z,y')\leq 8/\mu^2+1$ such that~\eqref{eq:rc:step2} holds
	for the vertex pair $(z,y')$. Since $\l$ and $s'(z,y')$ are odd and since
	$s'(z,y')\leq 8/\mu^2+1<\l$, for the odd integer
	\[
		s''=\l-s'(z,y')-1\geq 1
	\]
	there are $(\mu^2/4)^{s''}|U|^{s''}$ walks of length $s''$ from $y'$ to some
	vertex $y\in N_x\subseteq N_R(x)$, which then extends to a $z$-$x$-walk of
	length $\l$. In other words for every $y'\in S_x$ there are
	at least
	\[
		\frac{1}{36}\left(\frac{\mu}{2}\right)^{6s'(z,y')+4}|U|^{s'(z,y')-1}
		\times
		\left(\frac{\mu^2}{4}\right)^{s''}|U|^{s''}
		\geq
		\frac{1}{36}\left(\frac{\mu}{2}\right)^{6\l-2}|U|^{\l-2}
	\]
	$x$-$z$-walks of length $\l$ in $R$ passing through~$y'$. Repeating this argument for every vertex
	$y'\in  S_x$ leads by our choice of $\beta$ in~\eqref{eq:rc:betal} on first sight to
	$2\beta\ell |U|^{\l-1}$ $x$-$z$-walks of length~$\l$. However, each walk may be counted once
	for each of its interior vertices. Thus the total number of distinct $x$-$z$-walks arising this
	way is at least $2\beta |U|^{\l-1}$ and for sufficiently large $n$
	at least half of these walks are indeed paths of length $\l$.
	Since $x$ and $z$ were arbitrary this shows that
	$R$ is $(\beta,\l)$-robust and concludes the proof of Proposition~\ref{prop:robust}.
\end{proof}

We close this section with the observation that two graphs $R$ and $R'$
on the same vertex set, obtained by applications of Proposition~\ref{prop:robust},
must share quite a few edges. This will be essential in the proof of
Theorem~\ref{thm:main} as it asserts that any pair of robust subgraphs
from two link graphs share some edges.

\begin{prop} \label{prop:intersect}
	Let $V$ be a set of $n$ vertices and let $R=(U,E)$ and $R'=(U',E')$
	be graphs on vertex sets $U$, $U'\subseteq V$.
	If for some $\alpha>0$ we have
	\[
		|U|\geq \left(\frac{2}{3}+\frac{\alpha}{2}\right) n
		\qand
		|E|\geq \left(\frac{5}{9}+\frac{\alpha}{2}\right)\frac{n^2}{2}-\frac{(n-|U|)^2}{2}
	\]
	and
	\[
		|U'|\geq \left(\frac{2}{3}+\frac{\alpha}{2}\right) n
		\qand
		|E'|\geq \left(\frac{5}{9}+\frac{\alpha}{2}\right)\frac{n^2}{2}-\frac{(n-|U'|)^2}{2}
	\]
	then  $|E\cap E'|\ge \alpha n^2/2$.
\end{prop}
\begin{proof}
Define the real numbers $\rho$, $\rho'$, and $\eta$ by
\[
	|U|=\rho n, \quad |U'|=\rho' n, \qand |E \cap E'| = \eta \frac{n^2}2\,.
\]
The assumptions on the sizes of $U$ and $U'$ assert
\begin{equation}\label{eq:341}
	\rho,\, \rho'\in \bigl[\tfrac23+\tfrac\alpha2, 1\bigr]\,.
\end{equation}
Similarly, the assumptions on $|E|$ and $|E'|$ and the sieve formula yield
\begin{equation}\label{eq:342}
	|E\cup E'|\ge \left(\frac{10}9+\alpha-(1-\rho)^2-(1-\rho')^2-\eta\right)\frac{n^2}2\,.
\end{equation}
On the other hand, we have
\begin{align*}
	|E\cup E'|
	& \le
	\left|\binom{U}2\cup\binom{U'}2\right|
	=
	\binom{|U|}2+\binom{|U'|}2-\binom{|U\cap U'|}2\,.
\end{align*}
Now $|U\cap U'|\ge (\rho+\rho'-1)n$ and by~\eqref{eq:341} the expression $\rho+\rho'-1$
is positive, so
\[
	\big|E\cup E'\big| \le \left(\rho^2+(\rho')^2-(\rho+\rho'-1)^2\right)\frac{n^2}2\,.
\]
Together with~\eqref{eq:342} this gives
\[
	\rho^2+(\rho')^2-(\rho+\rho'-1)^2
	\ge
	\frac{10}9+\alpha-(1-\rho)^2-(1-\rho')^2-\eta\,,
\]
i.e.,
\[
	(\rho-\rho')^2 +\eta \ge \frac19+\alpha\,.
\]
But~\eqref{eq:341} implies $(\rho-\rho')^2<1/9$, and thus we have indeed
$\eta\ge\alpha$.
\end{proof}

\begin{cor}\label{cor:34}
	Given Setup~\ref{setup} we have $|E(R_u)\cap E(R_v)|\ge \alpha n^2/2$
	for all $u, v\in V$.
\end{cor}

\begin{proof}
	By Proposition~\ref{prop:robust}\ref{it:rc1} and~\ref{it:rc2} the graphs $R_u$
	and $R_v$ satisfy the assumptions of Proposition~\ref{prop:intersect}.
\end{proof}

\section{Connectable pairs}
\label{sec:connecting}

In this section we establish the Connecting Lemma (Proposition~\ref{lem:con}) and, therefore,
justify the notion of connectable pairs from Definition~\ref{def:connectable}
by showing that such pairs indeed can be connected by tight paths in~$H$.
\begin{proof}[Proof of Proposition~\ref{lem:con}] Let $\zeta>0$ be given and
set
\begin{equation}\label{eq:theta}
	\theta=\frac{1}{2}\left(\frac{2}{3}\right)^{\l^2-1}\left(\frac{\alpha\beta\zeta}{2}\right)^{\l+1}.
\end{equation}
Let $(x, y)$ and $(z, w)$ be two disjoint $\zeta$-connectable pairs of vertices.
We recall Definition~\ref{def:connectable}, set $t=\lceil \zeta n\rceil$, and let
\[
	\{u(1), \ldots, u(t)\}\subseteq U_{xy}
	\quad \text{ as well as } \quad
	\{v(1), \ldots, v(t)\}\subseteq U_{zw}
\]
be arbitrary $t$-subsets of $U_{xy}$ and $U_{zw}$, respectively.

Let us define
\[
	I_{ab}=\bigl\{i\in [t]\colon ab\in E(R_{u(i)}) \cap E(R_{v(i)})\bigr\}
\]
for any ordered pair $(a, b)$ of vertices from $V$. Then double counting shows that
\begin{equation}\label{eq:sumIab}
	\sum_{(a, b)\in V^2}|I_{ab}|=\sum_{i=1}^t \big|E(R_{u(i)}) \cap E(R_{v(i)})\big|\ge \frac\alpha 2  n^2t\,,
\end{equation}
where the last inequality follows from Corollary~\ref{cor:34}.
We intend to estimate the number $T$ of all  tight $(x, y)$-$(z, w)$-walks of the form
\begin{equation}\label{eq:con-path}
	xy u(i_1) r_1 r_2 u(i_2) \ldots r_{\ell-2} r_{\ell-1} u(i_{(\ell+1)/2)}) a b
	v(j_1) s_1 s_2 v(j_2) \ldots  s_{\ell^-2} s_{\ell-1} v(j_{(\ell+1)/2)}) z w\,,
\end{equation}
where tight walks are defined similarly like tight paths, but vertices are allowed to repeat.
Such walks can be represented by
 sextuples
\[
	\bigl(\seq{\imath}, \seq{\jmath}, \seq{r}, \seq{s}, a, b\bigr)\in
	[t]^{(\ell+1)/2}\times [t]^{(\ell+1)/2}\times V^{\ell-1}\times V^{\ell-1}\times V\times V.
\]

Intuitively, these walks connect $(x, y)$ to $(z, w)$
via an arbitrary ``middle pair'' $(a, b)$ (see Figure~\ref{fig:connect}).
The construction of such walks can be reduced
to a $2$-uniform problem in link graphs by demanding that for every
$k\in[(\ell+1)/2]$ we have:
\begin{enumerate}[label=\alabel]
\item\label{it:con1} $i_k, j_k\in I_{ab}$,
\item\label{it:con2}  $yr_1 \ldots r_{\ell-1}a$ is a path in $R_{u(i_k)}$,
\item\label{it:con3}  and $bs_1 \ldots s_{\ell-1}z$ is a path in $R_{v(j_k)}$.
\end{enumerate}

\begin{figure}[ht]
\centering
\begin{tikzpicture}[scale=1.05]

	\coordinate (x) at (0,0);
	\coordinate (y) at (0.9,0);
	
	\coordinate (r1) at (1.8,0.8);
	\coordinate (r2) at (2.6,1.7);
	\coordinate (r3) at (3.2,2.7);
	\coordinate (r4) at (3.7,3.8);
	
	\coordinate (r6) at (4.4,5.7);
	\coordinate (r7) at (5.1,6.5);

	\coordinate (a) at (6,7);
	\coordinate (b) at (7,7);
	
	\coordinate (s1) at (7.9,6.5);
	\coordinate (s2) at (8.6,5.7);
	\coordinate (s3) at (9.0,4.6);
	\coordinate (s4) at (9.5,3.3);
	
	\coordinate (s6) at (10.4,1.7);
	\coordinate (s7) at (11.2,0.8);

	\coordinate (z) at (12.1,0);
	\coordinate (w) at (13,0);

	\coordinate (u1) at (0.7,3.3);
	\coordinate (u2) at (1.8,5.7);
	\coordinate (u9) at (2.9,8.1);
	
	\coordinate (v1) at (10.1,8.1);
	\coordinate (v2) at (11.2,5.7);
	\coordinate (v9) at (12.3,3.3);

	\begin{pgfonlayer}{front}
		
		\foreach \i in {x, y, z, w, a, b, r1, r2, r3, r4, r6, r7, s1, s2, s3, s4, s6, s7}
			\fill  (\i) circle (2pt);
			
		\foreach \i in {u1, u2, u9, v1, v2, v9}{
			\draw[blue!75!black, very thick]  (\i) circle (2pt);
			\fill[blue!75!white]  (\i) circle (2pt);
		}
		
		\node at (0,-0.4) {\textcolor{green!60!black}{$x$}};
		\node at (0.9,-0.43) {\textcolor{green!60!black}{$y$}};
		\node at (6,6.58) {		{$a$}};
		\node at (7,6.63) {		{$b$}};
		\node at (12.1,-0.4) {\textcolor{green!60!black}{$z$}};
		\node at (13,-0.4) {\textcolor{green!60!black}{$w$}};
		
		\node at (2.1,0.5) {$r_1$};
		\node at (10.9,0.5) {$s_{\l-1}$};

		\node at (2.9,1.4) {$r_2$};
		\node at (10,1.4) {$s_{\l-2}$};
		
		\node at (4.9,5.5) {$r_{\l-2}$};
		\node at (8.3,5.5) {$s_{2}$};

		\node at (5.6,6.2) {$r_{\l-1}$};
		\node at (7.6,6.2) {$s_{1}$};
		
		\node at (3.6,2.5) {$r_3$};
		\node at (4.1,3.7) {$r_4$};
		
		\node at (8.7,4.4) {$s_3$};
		\node at (9.2,3.1) {$s_4$};
		
		\node at (1.0,6.5) {\textcolor{blue!75!black}{$U_{xy}$}};
		\node at (0.85,3.7) {\footnotesize \textcolor{blue!75!black}{$u(i_1)$}};
		\node at (1.95,6.1) {\footnotesize \textcolor{blue!75!black}{$u(i_2)$}};
		\node at (1.9,8.15) {\footnotesize \textcolor{blue!75!black}{$u(i_{\frac{\l+1}{2}})$}};
		
		\node at (12.0,6.5) {\textcolor{blue!75!black}{$U_{zw}$}};
		\node at (10.9,8.25) {\footnotesize \textcolor{blue!75!black}{$v(j_1)$}};
		\node at (11.25,6.1) {\footnotesize \textcolor{blue!75!black}{$v(j_2)$}};
		\node at (12.1,3.75) {\footnotesize \textcolor{blue!75!black}{$v(j_{\frac{\l+1}{2}})$}};

	\end{pgfonlayer}
	
	\begin{pgfonlayer}{background}
		\draw[green, line width=2pt] (x) -- (y);
		\draw[black!40!white, line width=2pt] (a) -- (b);
		\draw[green, line width=2pt] (z) -- (w);
		
		\draw[black!40!white, line width=2pt] (y) -- (r1);
		\draw[black!40!white, line width=2pt] (r1) -- (r2);
		\draw[black!40!white, line width=2pt] (r2) -- (r3);
		\draw[black!40!white, line width=2pt] (r3) -- (r4);
		
		\draw[black!40!white, line width=2pt, decorate, decoration={snake,segment length=9,amplitude=4}] (r4) -- (r6);
		
		\draw[black!40!white, line width=2pt] (r6) -- (r7);
		\draw[black!40!white, line width=2pt] (r7) -- (a);
		
		\draw[black!40!white, line width=2pt] (b) -- (s1);
		\draw[black!40!white, line width=2pt] (s1) -- (s2);
		\draw[black!40!white, line width=2pt] (s2) -- (s3);
		\draw[black!40!white, line width=2pt] (s3) -- (s4);
		
		\draw[black!40!white, line width=2pt, decorate, decoration={snake,segment length=9,amplitude=4}] (s4) -- (s6);
		
		\draw[black!40!white, line width=2pt] (s6) -- (s7);
		\draw[black!40!white, line width=2pt] (s7) -- (z);
	\end{pgfonlayer}
	
	\draw[rotate around={-26:(u2)},blue!75!black, line width=2pt] (u2) ellipse (18pt and 3.4cm);
	\fill[blue!75!white,opacity=0.2,rotate around={-26:(u2)}] (u2) ellipse (18pt and 3.4cm);
	
	\draw[rotate around={26:(v2)},blue!75!black, line width=2pt] (v2) ellipse (18pt and 3.4cm);
	\fill[blue!75!white,opacity=0.2,rotate around={26:(v2)}] (v2) ellipse (18pt and 3.4cm);
	
	\qedge{(u1)}{(y)}{(x)}{4.5pt}{1.5pt}{red!70!black}{red!70!black,opacity=0.2};
	\qedge{(u1)}{(r1)}{(y)}{4.5pt}{1.5pt}{red!70!black}{red!70!black,opacity=0.2};
	\qedge{(u1)}{(r2)}{(r1)}{4.5pt}{1.5pt}{red!70!black}{red!70!black,opacity=0.2};
	
	\qedge{(u2)}{(r2)}{(r1)}{4.5pt}{1.5pt}{red!70!black}{red!70!black,opacity=0.2};
	\qedge{(u2)}{(r3)}{(r2)}{4.5pt}{1.5pt}{red!70!black}{red!70!black,opacity=0.2};
	\qedge{(u2)}{(r4)}{(r3)}{4.5pt}{1.5pt}{red!70!black}{red!70!black,opacity=0.2};
	
	\qedge{(u9)}{(r7)}{(r6)}{4.5pt}{1.5pt}{red!70!black}{red!70!black,opacity=0.2};
	\qedge{(u9)}{(a)}{(r7)}{4.5pt}{1.5pt}{red!70!black}{red!70!black,opacity=0.2};
	\qedge{(u9)}{(b)}{(a)}{4.5pt}{1.5pt}{red!70!black}{red!70!black,opacity=0.2};
	
	\qedge{(v1)}{(b)}{(a)}{4.5pt}{1.5pt}{red!70!black}{red!70!black,opacity=0.2};
	\qedge{(v1)}{(s1)}{(b)}{4.5pt}{1.5pt}{red!70!black}{red!70!black,opacity=0.2};
	\qedge{(v1)}{(s2)}{(s1)}{4.5pt}{1.5pt}{red!70!black}{red!70!black,opacity=0.2};
	
	\qedge{(v2)}{(s2)}{(s1)}{4.5pt}{1.5pt}{red!70!black}{red!70!black,opacity=0.2};
	\qedge{(v2)}{(s3)}{(s2)}{4.5pt}{1.5pt}{red!70!black}{red!70!black,opacity=0.2};
	\qedge{(v2)}{(s4)}{(s3)}{4.5pt}{1.5pt}{red!70!black}{red!70!black,opacity=0.2};
	
	\qedge{(v9)}{(w)}{(z)}{4.5pt}{1.5pt}{red!70!black}{red!70!black,opacity=0.2};
	\qedge{(v9)}{(z)}{(s7)}{4.5pt}{1.5pt}{red!70!black}{red!70!black,opacity=0.2};
	\qedge{(v9)}{(s7)}{(s6)}{4.5pt}{1.5pt}{red!70!black}{red!70!black,opacity=0.2};
	
\end{tikzpicture}
\caption{Connecting the \textcolor{green!60!black}{$\zeta$-connectable pairs $(x,y)$}
	and \textcolor{green!60!black}{$(z,w)$} through
	middle pair $(a,b)$
	using vertices from the sets \textcolor{blue!75!black}{$U_{xy}$} and \textcolor{blue!75!black}{$U_{zw}$}.}
\label{fig:connect}
\end{figure}
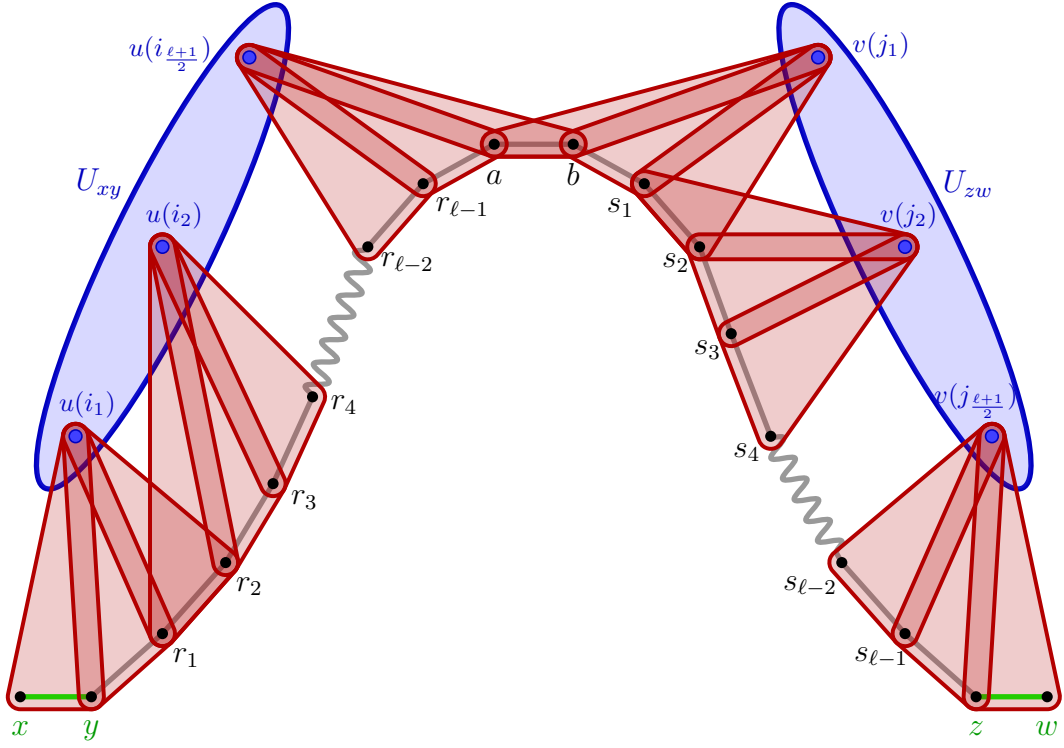

In other words, if $T^*$ denotes the number of sextuples
$\bigl(\seq{\imath}, \seq{\jmath}, \seq{r}, \seq{s}, a, b\bigr)$ satisfying
the conditions~\ref{it:con1},~\ref{it:con2}, and~\ref{it:con3}, then
$T\ge T^*$. Note that the hyperedges $xyu(i_1)$ and $v(j_{(\l+1)/2})zw$ are not forced by~\ref{it:con1}
and~\ref{it:con2}, but are a direct consequence of $u(i_1)\in U_{xy}$ and $v(j_{(\l+1)/2})\in U_{zw}$.
Similarly, the required hyperedges $u(i_{(\l+1)/2})ab$ and $abv(j_1)$ are a consequence of~\ref{it:con1}.
On the other hand, conditions~\ref{it:con1}\,--\,\ref{it:con3} imply several additional hyperedges, which are not
required for the $(x,y)$-$(z,w)$-walk. Hence, indeed we have $T\geq T^*$.
Below we shall show
\begin{equation}\label{eq:sigma-low}
	T^*\ge 2\theta n^{3\ell+1}\,.
\end{equation}
Since at most $O(n^{3\ell})$ of the corresponding walks~\eqref{eq:con-path}
can fail to be a path (due to the presence
of repeated vertices), this trivially implies Proposition~\ref{lem:con}.

As a first step towards the proof of~\eqref{eq:sigma-low} we will fix for a while
the middle vertices $a$ and $b$ and study the number $T_{ab}$ of possibilities
to complete a walk of the desired kind by an appropriate choice of the $3\ell-1$
remaining vertices. Evidently
\begin{equation}\label{eq:sigma-prod}
	T_{ab}=R_{ab}S_{ab}\,,
\end{equation}
where $R_{ab}$ denotes the number of possibilities to choose
$i_1, \ldots, i_{(\ell+1)/2}\in I_{ab}$ and vertices $r_1, \ldots, r_{\ell-1}\in V$
such that \ref{it:con2} holds and $S_{ab}$ has a similar meaning with respect to
the numbers $j_k$, the vertices $s_k$, and property \ref{it:con3}. Given any sequence
$\seq{r}=(r_1, \ldots, r_{\ell-1})\in V^{\ell-1}$ of vertices, we set
\[
	D(\seq{r}\,)=\big\{i\in  I_{ab}\colon y\seq{r}a \text{ is a path in } R_{u(i)}\big\}\,.
\]
Then
\[
	R_{ab}=\sum_{\seq{r}\in V^{\ell-1}}|D(\seq{r}\,)|^{(\ell+1)/2}
\]
and from the $(\beta, \ell)$-robustness of $R_{u(i)}$ combined with property~\ref{it:rc1}
of Proposition~\ref{prop:robust} applied for every $i\in I_{ab}$,
we infer by means of double counting that
\[
	\sum_{\seq{r}\in V^{\ell-1}}|D(\seq{r}\,)|\ge |I_{ab}|\cdot \beta \left(\frac{2}{3}n\right)^{\ell-1}\,.
\]
Thus a standard convexity argument shows
\[
	R_{ab}\ge n^{\ell-1}|I_{ab}|^{(\ell+1)/2}\beta^{(\ell+1)/2}\left(\frac{2}{3}\right)^{(\l^2-1)/2}\,.
\]
Applying this argument also to $S_{ab}$, using~\eqref{eq:sigma-prod}, and
summing over all $(a,b)\in V^2$ we deduce
\begin{align*}
	T^*
	=
	\sum_{(a, b)\in V^2}T_{ab}
	&\overset{\phantom{\eqref{eq:sumIab}}}{\ge}
	\left(\frac{2}{3}\right)^{\l^2-1}\beta^{\ell+1}n^{2(\ell-1)} \times \sum_{(a, b)\in V^2} |I_{ab}|^{\ell+1}\\
	&\overset{\eqref{eq:sumIab}}{\geq}
	\left(\frac{2}{3}\right)^{\l^2-1}\beta^{\ell+1}n^{2(\ell-1)} \times \left(\frac\alpha 2 t\right)^{\l+1} n^2
	\,,
\end{align*}
where we used Jensen's inequality in the last step.
Recalling the choice of $\theta$ in~\eqref{eq:theta}
and that $t=\lceil\zeta n\rceil$
entails~\eqref{eq:sigma-low} and this concludes the proof.
\end{proof}

We close this section with the following immediate consequence of Definition~\ref{def:connectable},
which we shall use at several occasions in the subsequent sections.

\begin{fact}\label{fact:easy}
Given Setup~\ref{setup} and $\zeta>0$, there are at most~$\zeta n^3$ triples
$(x, y, z)\in V^3$ with $xy\in E(R_z)$ such that the pair $xy$ fails to
be $\zeta$-connectable.
\end{fact}
\begin{proof}
If an (unordered) pair $xy$ fails to be $\zeta$-connectable, then it follows from Definition~\ref{def:connectable}
that $|U_{xy}|\leq \zeta n$ and, hence, $xy$ is an edge in $R_z$ for at most $\zeta |V|=\zeta n$ vertices~$z\in V$.
Since there are at most $n^2$ ordered pairs $(x,y)\in V^2$, the fact follows.
\end{proof}

\section{Reservoir}\label{sec:reservoir}
In this section we focus on the Reservoir lemma (Proposition~\ref{prop:reservoir}).
The existence of such a reservoir set is established by a
standard probabilistic argument.

\begin{proof}[Proof of Proposition~\ref{prop:reservoir}]
	Consider a random subset $\cR\subseteq V$ with elements included
	independently with probability
	\[
		p=\left(1-\frac{1}{10\l}\right)\theta^2_{*}\,.
	\]
	Consequently,~$|\cR|$ is binomially distributed and we infer from
	Chernoff's inequality that
	\begin{equation}\label{eq:sizeRa}
		\PP\big(|\cR|<\theta_*^2n/2\big)=
		o(1)\,.
	\end{equation}
	Moreover, since $\theta^2n\geq(4/3)^{\frac{1}{3\l+1}}pn\geq(1+c)\EE[|\cR|]$ for some sufficiently small $c=c(\l)>0$, we
	also have
	\begin{equation}\label{eq:sizeRb}
		\PP\big(|\cR|>\theta_*^2n\big)
		\leq
		\PP\Big(|\cR|>(4/3)^{\frac{1}{3\l+1}}pn\Big)
		=
		o(1)\,.
	\end{equation}
	Recall that for every disjoint pair $(x,y)$ and $(z,w)$ of $\zeta_{**}$-connectable pairs
	Proposition~\ref{lem:con} ensures the existence of at least $\theta_{**}n^{3\l+1}$
	tight $(x,y)$-$(z,w)$-paths of length $3(\l+1)$ (having $3\l+5$ vertices in total). Let
	$X=X((x,y),(z,w))$ be a random variable counting the number of $(x,y)$-$(z,w)$-paths
	with all $3\l+1$ internal vertices in $\cR$.
	Consequently
	\begin{equation}\label{eq:EEX}
		\EE[X]
		\geq
		p^{3\l+1}\cdot\theta_{**}n^{3\l+1}\,.
	\end{equation}
	Since including or not including a particular vertex into $\cR$ affects the random variable~$X$
	by at most $(3\l+1) n^{3\l}$, the Azuma--Hoeffding inequality (see, e.g.,~\cite{JLR00}*{Corollary~2.27})
	asserts
	\begin{align}
		\PP\big(X\leq \tfrac{2}{3}\theta_{**}(pn)^{3\l+1}\big)
		&\overset{\eqref{eq:EEX}}{\leq}
		\PP\big(X\leq \tfrac{2}{3}\EE[X]\big)\nonumber\\
		&\overset{\phantom{\eqref{eq:EEX}}}{\leq}
		\exp\left(-\frac{\EE[X]^2}{18\cdot n\cdot ((3\l+1)n^{3\l})^2}\right)=\exp\big(-\Omega(n)\big)\,.\label{eq:RAzuma}
	\end{align}
	Since there are at most $n^4$  pairs of $\zeta_{**}$-connectable pairs that we have to consider,
	in view of~\eqref{eq:sizeRb}, the union bound combined with~\eqref{eq:RAzuma},
	implies that a.a.s.\ the set $\cR$ has the property that
	for every pair of connectable
	pairs at least  $\theta_{**}|\cR|^{3\l+1}/2$ tight connecting paths
	have all internal vertices in~$\cR$. In addition, due to~\eqref{eq:sizeRa} and~\eqref{eq:sizeRb}
	a.a.s.\ the set~$\cR$ also satisfies $\theta^2_*n/2\leq|\cR|\leq\theta_*^2n$. Consequently,
	a reservoir set $\cR$ with all required properties indeed exists.
\end{proof}

In Section~\ref{sec:longpath} we will frequently need to connect $\zeta_{**}$-connectable pairs
through the reservoir. Whenever such a connection is made, the part of the reservoir that may
still be used for further connections shrinks by $3\ell+1$ vertices. Although $\Omega(n)$ such
connections are needed, we shall be able to keep the reservoir almost intact throughout this
process, which in turn guarantees that there will always be some permissible connections left.

\begin{lemma}\label{lem:use-reservoir}
		Given Setup~\ref{setup2} with a reservoir set $\cR\subseteq V$,
		let $\cR'\subseteq \cR$ be an arbitrary subset of size at most $2\theta^2_{**}n$.
		Then for all disjoint pairs of $\zeta_{**}$-connectable pairs~$(x,y)$ and~$(z,w)$
		there is a tight $(x, y)$-$(z, w)$-path of length $3(\ell+1)$ in $H$ whose internal vertices
		belong to $\cR\setminus \cR'$.
\end{lemma}

\begin{proof}
	Recalling $|\cR|\geq \theta_{*}^2n/2$ and the hierarchy~\eqref{eq:consthierall}
	yields $|\cR'|\leq 2\theta^2_{**}n\leq \frac{\theta_{**}}{8\l}|\cR|$. Moreover,
	every given vertex in $\cR'$ is an internal vertex of at most $(3\l+1)|\cR|^{3\l}$ tight
	$(x, y)$-$(z, w)$-paths of length $3(\ell+1)$ in $H$ whose internal vertices
	belong to~$\cR$. Consequently,  there are still at least
	\[
		\frac{\theta_{**}}{2}|\cR|^{3\l+1}-|\cR'|\cdot (3\l+1)|\cR|^{3\l}
		\geq
		\frac{\theta_{**}}{2}|\cR|^{3\l+1}-\frac{\theta_{**}}{8\l}\cdot (3\l+1)|\cR|^{3\l+1}
		>
		0
	\]
	such paths with all internal vertices in~$\cR\setminus\cR'$.
\end{proof}

\section{Absorbing path}\label{sec:absorbing}
In this section we prove Proposition~\ref{prop:absorbingP}, that is, we establish the existence of an absorbing path. The
following special hypergraph (the so-called $v$-absorber, see Figure~\ref{sfig:absA})
will allow us to absorb a given vertex~$v$ into a path containing a $v$-absorber (see Figure~~\ref{sfig:absB}).

\begin{dfn}\label{dfn-vabs} Given Setup~\ref{setup2} and a vertex $v\in V$,
a $9$-tuple $(a, b, c, d, z, x, y, y', x')\in (V\setminus\{v\})^9$ of distinct vertices such that
\begin{enumerate}[label=\rmlabel]
	\item\label{it:vabs1}   $zab, zbc, zcd, zxy, zyy', zy'x', xyy', yy'x'\in E$, and
	\item\label{it:vabs2}  the pairs $ab$, $cd$, $xy$, and $y'x'$
	are $\zeta_*$-connectable
\end{enumerate}
is called  \emph{a $v$-absorber} if, in addition, $vab$, $vbc$, $vcd\in E$.
\end{dfn}

\begin{figure}[ht]
\centering
\begin{subfigure}[b]{0.55\textwidth}
\begin{tikzpicture}[scale=1.2]
	\coordinate (v) at (0,0);
	
	\coordinate (x1) at (1.4,-2.4);
	\coordinate (x2) at (2,-0.8);
	\coordinate (x3) at (2,0.8);
	\coordinate (x4) at (1.4,2.4);
	
	\coordinate (z) at (4,0);
	
	\coordinate (y1) at (5.4,-2.4);
	\coordinate (y2) at (6,-0.8);
	\coordinate (y3) at (6,0.8);
	\coordinate (y4) at (5.4,2.4);

	\begin{pgfonlayer}{front}
		\fill (v) circle (2pt);
		\fill (z) circle (2pt);
		
		\foreach \i in {1,2,3,4}{
			\fill  (x\i) circle (2pt);
			\fill  (y\i) circle (2pt);
		}

		\node at ($(v)+(180:0.32)$) {\footnotesize $v$};
		
		\node at ($(x2)+(-50:0.34)$) {\footnotesize $b$};
		\node at ($(x3)+(45:0.34)$) {\footnotesize $c$};
		\node at ($(x1)+(-25:0.33)$) {\footnotesize $a$};
		\node at ($(x4)+(25:0.36)$) {\footnotesize $d$};
		
		\node at ($(z)+(100:0.57)$) {\footnotesize $z$};
		
		\node at ($(y1)+(-10:0.39)$) {\footnotesize $x'$};
		\node at ($(y2)+(-5:0.39)$) {\footnotesize $y'$};
		\node at ($(y3)+(5:0.39)$) {\footnotesize $y$};
		\node at ($(y4)+(10:0.39)$) {\footnotesize $x$};
		
	\end{pgfonlayer}
	
	\begin{pgfonlayer}{background}
		\draw[green, line width=2pt] (x1) -- (x2);
		\draw[green, line width=2pt] (x3) -- (x4);
		\draw[green, line width=2pt] (y1) -- (y2);
		\draw[green, line width=2pt] (y3) -- (y4);
	\end{pgfonlayer}

	\qedge{(v)}{(x2)}{(x1)}{4.5pt}{1.5pt}{red!70!white}{red!70!white,opacity=0.2};
	\qedge{(v)}{(x3)}{(x2)}{4.5pt}{1.5pt}{red!70!white}{red!70!white,opacity=0.2};
	\qedge{(v)}{(x4)}{(x3)}{4.5pt}{1.5pt}{red!70!white}{red!70!white,opacity=0.2};
	\qedge{(z)}{(x1)}{(x2)}{4.5pt}{1.5pt}{red!70!black}{red!70!black,opacity=0.2};
	\qedge{(z)}{(x2)}{(x3)}{4.5pt}{1.5pt}{red!70!black}{red!70!black,opacity=0.2};
	\qedge{(z)}{(x3)}{(x4)}{4.5pt}{1.5pt}{red!70!black}{red!70!black,opacity=0.2};
	\qedge{(z)}{(y2)}{(y1)}{6.5pt}{1.5pt}{red!70!white}{red!70!white,opacity=0.2};
	\qedge{(z)}{(y3)}{(y2)}{6.5pt}{1.5pt}{red!70!white}{red!70!white,opacity=0.2};
	\qedge{(z)}{(y4)}{(y3)}{6.5pt}{1.5pt}{red!70!white}{red!70!white,opacity=0.2};
	\qedge{(y1)}{(y3)}{(y2)}{4.0pt}{1.5pt}{red!70!black}{red!70!black,opacity=0.2};
	\qedge{(y4)}{(y3)}{(y2)}{4.0pt}{1.5pt}{red!70!black}{red!70!black,opacity=0.2};
	\end{tikzpicture}
	\caption{$v$-absorber with all hyperedges}
	\label{sfig:absA}
\end{subfigure}\begin{subfigure}[b]{0.45\textwidth}
\begin{tikzpicture}[scale=1.1]
	
	\draw[opacity=0] (-3.2,7.1) rectangle (3.2,13.3);

	\coordinate (v1) at (97:13);
	
	\coordinate (a1) at (101:12);
	\coordinate (b1) at (99:12);
	\coordinate (c1) at (95:12);
	\coordinate (d1) at (93:12);
	
	\coordinate (z1) at (97:12);
	
	\coordinate (x1) at (86:12);
	\coordinate (y1) at (84:12);
	\coordinate (yp1) at (82:12);
	\coordinate (xp1) at (80:12);
	
	\draw[very thick, decorate, decoration={snake,segment length=10,amplitude=3}, red!70!black]
		(x1) arc [start angle=86,radius=12, delta angle=7];
		
	\draw[very thick, decorate, decoration={snake,segment length=10,amplitude=3}, red!70!black]
		(a1) arc [start angle=101,radius=12, delta angle=4];
		
	\draw[very thick, decorate, decoration={snake,segment length=10,amplitude=3}, red!70!black]
		(xp1) arc [start angle=80,radius=12, delta angle=-5];
	
	\begin{pgfonlayer}{front}
		\foreach \i in {v1,a1,b1,c1,d1,z1,x1,y1,xp1,yp1}
			\fill (\i) circle (2pt);
		\node at (97:13.22) {\footnotesize $v$};
		
		\node at (101:12.32) {\footnotesize $a$};
		\node at (99:12.35) {\footnotesize $b$};
		\node at (95:12.32) {\footnotesize $c$};
		\node at (93:12.35) {\footnotesize $d$};
		
		\node at (97:11.57) {\footnotesize $z$};
		
		\node at (86:12.32) {\footnotesize $x$};
		\node at (84:12.35) {\footnotesize $y$};
		\node at (81.5:12.41) {\footnotesize $y'$};
		\node at (79.5:12.38) {\footnotesize $x'$};
		
	\end{pgfonlayer}
	
	\qedge{(z1)}{(a1)}{(b1)}{4.5pt}{1.5pt}{red!70!black}{red!70!black,opacity=0.2};
	\qedge{(z1)}{(b1)}{(c1)}{4.5pt}{1.5pt}{red!70!black}{red!70!black,opacity=0.2};
	\qedge{(z1)}{(c1)}{(d1)}{4.5pt}{1.5pt}{red!70!black}{red!70!black,opacity=0.2};
	
	\qedge{(yp1)}{(x1)}{(y1)}{4.5pt}{1.5pt}{red!70!black}{red!70!black,opacity=0.2};
	\qedge{(xp1)}{(y1)}{(yp1)}{4.5pt}{1.5pt}{red!70!black}{red!70!black,opacity=0.2};
	
	\begin{pgfonlayer}{background}
		\draw[green, line width=2pt] (a1) -- (b1);
		\draw[green, line width=2pt] (c1) -- (d1);
		\draw[green, line width=2pt] (x1) -- (y1);
		\draw[green, line width=2pt] (yp1) -- (xp1);
	\end{pgfonlayer}
	
	\draw[black!50!white, line width=1.5pt, line cap=round, shorten <=3.8pt,shorten >=6.5pt,->] (v1) -- (z1);
	\draw[black!50!white, line width=1.5pt, line cap=round, ->] (96.2:11.55) -| (83:11.78);
	
	\coordinate (v2) at ($(97:12)-(0,3.2)$);
	
	\coordinate (a2) at ($(101:12)-(0,3.2)$);
	\coordinate (b2) at ($(99:12)-(0,3.2)$);
	\coordinate (c2) at ($(95:12)-(0,3.2)$);
	\coordinate (d2) at ($(93:12)-(0,3.2)$);
	
	\coordinate (z2) at ($(83:12)-(0,3.2)$);
	
	\coordinate (x2) at ($(87:12)-(0,3.2)$);
	\coordinate (y2) at ($(85:12)-(0,3.2)$);
	\coordinate (yp2) at ($(81:12)-(0,3.2)$);
	\coordinate (xp2) at ($(79:12)-(0,3.2)$);
	
	\draw[very thick, decorate, decoration={snake,segment length=10,amplitude=3}, red!70!black]
		(x2) arc [start angle=86,radius=12, delta angle=6];
		
	\draw[very thick, decorate, decoration={snake,segment length=10,amplitude=3}, red!70!black]
		(a2) arc [start angle=101,radius=12, delta angle=4];
		
	\draw[very thick, decorate, decoration={snake,segment length=10,amplitude=3}, red!70!black]
		(xp2) arc [start angle=79,radius=12, delta angle=-4];
	
	\begin{pgfonlayer}{front}
		\foreach \i in {v2,a2,b2,c2,d2,z2,x2,y2,xp2,yp2}
			\fill (\i) circle (2pt);
			
		\node at ($(97:12)-(0,3.55)$) {\footnotesize $v$};

		\node at ($(101:12)-(0,2.87)$) {\footnotesize $a$};
		\node at ($(99:12)-(0,2.85)$) {\footnotesize $b$};
		\node at ($(95:12)-(0,2.88)$) {\footnotesize $c$};
		\node at ($(93:12)-(0,2.85)$) {\footnotesize $d$};
		
		\node at ($(83.2:12)-(0,3.55)$) {\footnotesize $z$};
		
		\node at ($(87:12)-(0,2.88)$) {\footnotesize $x$};
		\node at ($(85:12)-(0,2.85)$) {\footnotesize $y$};
		\node at ($(80.5:12)-(0,2.79)$) {\footnotesize $y'$};
		\node at ($(78.5:12)-(0,2.82)$) {\footnotesize $x'$};
		
	\end{pgfonlayer}
	
	\qedge{(v2)}{(a2)}{(b2)}{4.5pt}{1.5pt}{red!70!white}{red!70!white,opacity=0.2};
	\qedge{(v2)}{(b2)}{(c2)}{4.5pt}{1.5pt}{red!70!white}{red!70!white,opacity=0.2};
	\qedge{(v2)}{(c2)}{(d2)}{4.5pt}{1.5pt}{red!70!white}{red!70!white,opacity=0.2};
	
	\qedge{(z2)}{(x2)}{(y2)}{4.5pt}{1.5pt}{red!70!white}{red!70!white,opacity=0.2};
	\qedge{(xp2)}{(z2)}{(yp2)}{4.5pt}{1.5pt}{red!70!white}{red!70!white,opacity=0.2};
	\qedge{(yp2)}{(z2)}{(y2)}{4.5pt}{1.5pt}{red!70!white}{red!70!white,opacity=0.2};
	
	\begin{pgfonlayer}{background}
		\draw[green, line width=2pt] (a2) -- (b2);
		\draw[green, line width=2pt] (c2) -- (d2);
		\draw[green, line width=2pt] (x2) -- (y2);
		\draw[green, line width=2pt] (yp2) -- (xp2);
	\end{pgfonlayer}
	\end{tikzpicture}
	\caption{$v$-absorber \textcolor{red!60!black}{before}/\textcolor{red!60!white}{after} absorption}
	\label{sfig:absB}
    \end{subfigure}
\caption{A $v$-absorber, where the \textcolor{green!60!black}{$\zeta_*$-connectable pairs} are indicated in green,
	hyperedges used  \textcolor{red!60!black}{before absorption} of $v$ are dark red and
	hyperedges used \textcolor{red!60!white}{after absorption} of $v$ are light red.}
\label{fig:absorber}
\end{figure}
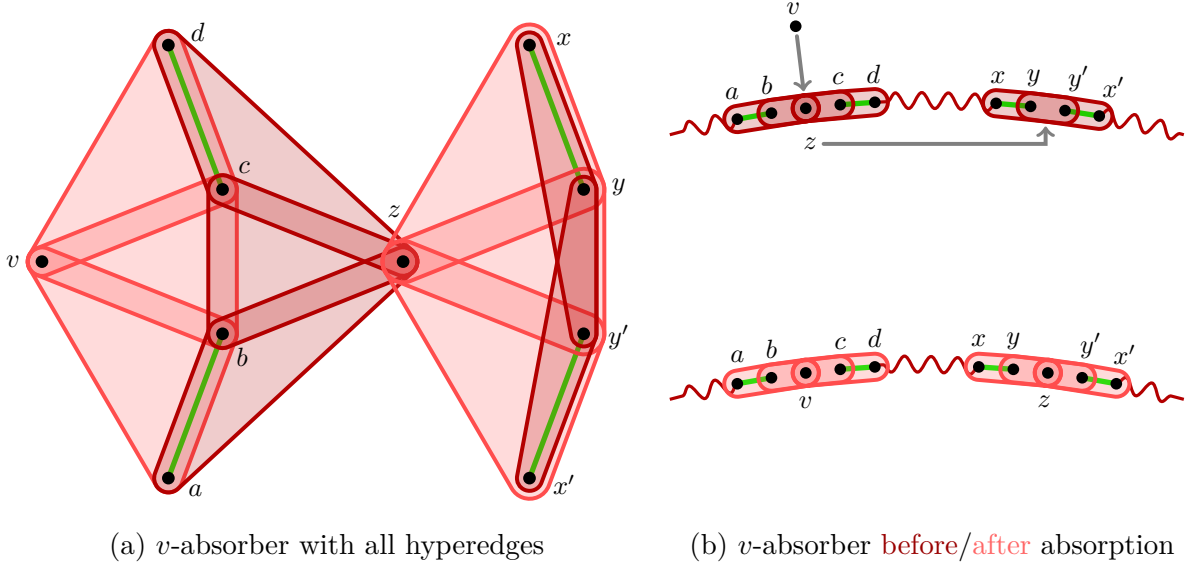
An important property of these configurations proved in
Lemma~\ref{lem:count-absorbers} below asserts
that for every vertex $v\in V$ there exist $\Omega(n^9)$ such $v$-absorbers.
For standard probabilistic reasons this will lead us to a family $\cF$ of
$\Omega(n)$ set-wise mutually disjoint $9$-tuples containing for each $v\in V$ at least
$\Omega(n)$ absorbers (see Lemma~\ref{lem:choose-F}).
Owing to condition~\ref{it:vabs2} of Definition \ref{dfn-vabs},
we may then use the Connecting Lemma (Proposition~\ref{lem:con})
for connecting them, i.e., for producing an {\it absorbing path} $P_A$ of length $\Omega(n)$,
which contains for every $v$-absorber $(a, b, c, d, z, x, y, y', x')\in \cF$
the subpaths $abzcd$ and $xyy'x'$.
If at the end of the proof of Theorem~\ref{thm:main} the need to absorb $v$ arises,
we shall simply replace in $P_A$, for one such $v$-absorber, the subpaths $abzcd$ and $xyy'x'$ by $abvcd$ and $xyzy'x'$ (see Figure~\ref{sfig:absB}).

Towards the goal of estimating the number of $v$-absorbers from below, we shall at first
only deal with  configurations consisting of the five vertices $z$, $x$, $y$, $y'$, and $x'$.
In the lemma that follows we do not pay attention to connectability demands yet.
For potential future references we point out that its proof  requires only a less restrictive minimum degree
condition than the one provided by Theorem~\ref{thm:main}.

\begin{lemma}~\label{lem:28}
	For every hypergraph $H=(V, E)$ with $n$ vertices and $\delta(H)\ge \frac{6}{11}\cdot\frac{n^2}2$
	there exist at least $n^5/28^4$ quintuples $(x, y, y', x', z)\in V^5$
	with the following properties:
	\begin{enumerate}[label=\rmlabel]
	 \item\label{it:labs1} $xyz, yy'z, x'y'z, xyy', yy'x'\in E$;
	\item\label{it:labs2} $d(y, z)> \frac{5}{12}n$.	\end{enumerate}
\end{lemma}

\begin{proof}
	We consider the function $f\colon E\to \RR$ defined by
	\[
		f(x,y,z)=\frac{n}{d(x,y)}+\frac{n}{d(x,z)}+\frac{n}{d(y,z)}
	\]
	and note that by double counting we have
	\begin{equation}\label{eq:trick}
		\sum_{xyz\in E}f(x,y,z)=n\cdot |\partial H|\leq \frac{n^3}2\,,
	\end{equation}
	where $\partial H$ denotes the set of those pairs in $V^{(2)}$
	that are contained in at least one edge of~$H$.
	An edge $e\in E(H)$ is said to be {\it central} if $f(e)\le \frac{28}{5}$.
	In view of~\eqref{eq:trick} the set $C$ of central edges satisfies
	$\tfrac{28}{5}|E\setminus C|\le \tfrac{n^3}{2}$, i.e.,
	$|E\setminus C|\le \tfrac{5}{56}n^3$.
	On the other hand, the minimum degree condition
	imposed on $H$ yields $|E|\ge \tfrac{1}{11}n^3$ and thus we have
	\begin{equation}\label{count-central}
		|C|=|E|-|E\setminus C|\ge \frac{n^3}{11}-\frac{5n^3}{56}=\frac{n^3}{11\cdot 56}
		>\frac{n^3}{28^2}\,.
	\end{equation}
	Next we will show the following statement.
	\begin{claim}
	If $yy'z$ is a central edge with
	\begin{equation}\label{eq:central-order}
		d(y, y')\ge d(y, z)\ge d(y', z)\,,
	\end{equation}
	then
	\begin{equation}\label{eq:xx'}
	|N(y, z)\cap N(y, y')|\ge \frac{n}{28}, \quad
	|N(y', z)\cap N(y, y')|\ge \frac{n}{28}\,,
	\end{equation}
	and~\ref{it:labs2} of Lemma~\ref{lem:28} holds.
	\end{claim}
	
	\begin{proof}
	Let $yy'z$ be a central edge satisfying~\eqref{eq:central-order}.
	Due to $f(y, y', z)\le \tfrac{28}{5}$ we have
	\[
		\frac{2n}{d(y, y')}+\frac n{d(y', z)}\le\frac{28}{5}\,.
	\]
	Moreover, the Cauchy--Schwarz inequality yields
	\[
		\left(\frac2{d(y, y')}+\frac1{d(y', z)}\right)\bigl(d(y, y')+d(y', z)\bigr)
		\ge (\sqrt{2}+1)^2>\frac{29}{5}\,.
	\]
	Hence
$$	d(y, y')+d(y, z)\ge d(y, y')+d(y', z)>\tfrac{29}{28}n,$$ which implies~\eqref{eq:xx'}.

	Finally, $f(y, y', z)\le \tfrac{28}{5}$, $d(y,y')\le n$, and~\eqref{eq:central-order} lead to
	\[
		\frac{28}5
		\ge \frac{n}{d(y, y')}+\frac{n}{d(y, z)}+\frac{n}{d(y', z)}
		\ge 1+\frac{2n}{d(y, z)}\,,
	\]
	which proves~\ref{it:labs2} of Lemma~\ref{lem:28}.
	\end{proof}
	
	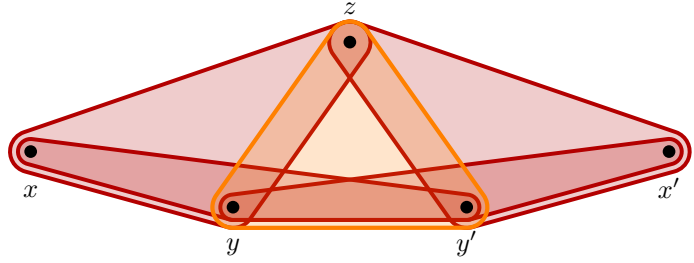
\begin{figure}[ht]
	\centering
	\begin{tikzpicture}[scale=1.2]
		
	\coordinate (z) at (0,-3cm);
	
	\coordinate (y1) at (230:5.5cm);
	\coordinate (y2) at (255:5cm);
	\coordinate (y3) at (285:5cm);
	\coordinate (y4) at (310:5.5cm);

	\begin{pgfonlayer}{front}
		\fill (z) circle (2pt);
		
		\foreach \i in {1,2,3,4}
			\fill  (y\i) circle (2pt);
			
		\node at ($(z)+(90:0.38)$) {\footnotesize $z$};
		
		\node at ($(y1)+(-90:0.44)$) {\footnotesize $x$};
		\node at ($(y2)+(-90:0.44)$) {\footnotesize $y$};
		\node at ($(y3)+(-90:0.40)$) {\footnotesize $y'$};
		\node at ($(y4)+(-90:0.40)$) {\footnotesize $x'$};
		
	\end{pgfonlayer}
	
	\qedge{(z)}{(y2)}{(y1)}{6.5pt}{1.5pt}{red!70!black}{red!70!black,opacity=0.2};
	\qedge{(z)}{(y4)}{(y3)}{6.5pt}{1.5pt}{red!70!black}{red!70!black,opacity=0.2};
	\qedge{(y1)}{(y3)}{(y2)}{4.0pt}{1.5pt}{red!70!black}{red!70!black,opacity=0.2};
	\qedge{(y4)}{(y3)}{(y2)}{4.0pt}{1.5pt}{red!70!black}{red!70!black,opacity=0.2};
	\qedge{(z)}{(y3)}{(y2)}{6.5pt}{1.6pt}{red!50!yellow}{red!50!yellow,opacity=0.2};
	\end{tikzpicture}
	\caption{Quintuple $(x,y,y',x',z)$ from Lemma~\ref{lem:28}
		with \textcolor{red!50!yellow}{central} edge
		\textcolor{red!50!yellow}{$zyy'$}.}
	\label{fig:central}
	\end{figure}
	
	Having thus established the above claim, we continue with the proof of Lemma~\ref{lem:28}
	(see Figure~\ref{fig:central}).
	To this end we remark that for every central edge~$yy'z$
	satisfying~\eqref{eq:central-order} the estimates~\eqref{eq:xx'} imply that there
	are at least $\tfrac{n}{28}$ choices of~$x$ and at least
	  $\tfrac{n}{28}$ choices of~$x'$ such that~\ref{it:labs1} of Lemma~\ref{lem:28} holds.
	Applying this argument to all central edges and taking~\eqref{count-central}
	into account we deduce the existence of
	at least $\frac{n^5}{28^4}$ quintuples $(x, y, y', x', z)\in V^5$
	with the desired properties.
\end{proof}

Next we prove that there are $\Omega(n)$ vertices  which are capable of playing
the r\^ole of $z$ in many absorbers.

\begin{dfn}\label{dfn:absorbable} Given Setup~\ref{setup2}, a
vertex $z\in V$ is said to be \emph{absorbable}
if there exist at least $\tfrac{n^4}{2^{21}}$ quadruples $(x, y, y', x')\in V^4$
such that
\begin{enumerate}[label=\alabel]
	\item\label{it:abs1}  the five triples
	$xyz$, $yzy'$, $zy'x'$, $xyy'$, and $yy'x'$ belong to~$E$,
	\item\label{it:abs2} and the pairs $xy$, $y'x'$ are $\zeta_*$-connectable.
\end{enumerate}
\end{dfn}

\begin{lemma}\label{lem:number-abs}
Given Setup~\ref{setup2},
there exist at least $\tfrac{n}{2^{21}}$ absorbable vertices.
\end{lemma}

\begin{proof}
Let $A\subseteq V^5$ denote the set of all quintuples $(x, y, y', x', z)$ satisfying
the conclusion of Lemma \ref{lem:28}, which then states that
\[
	|A|\ge \frac{n^5}{28^4}\,.
\]

We intend to show that for ``most'' of these quintuples the pairs
$xy$ and $x'y'$ are \mbox{$\zeta_*$-connectable}. This will then imply
Lemma~\ref{lem:number-abs} in view of an easy counting argument.

As we shall verify below, the following
three sets of ``exceptional'' quintuples are small:
\begin{align*}
	Q_1 &=\{(x, y, y', x', z)\in A\colon y\not\in V(R_z)\}\,, \\
	Q_2 &=\{(x, y, y', x', z)\in A\colon y\in V(R_z) \text{ but }
			\{x, x', y'\}\not\subseteq V(R_z)\}\,, \\
\intertext{and}
    Q_3 &=\{(x, y, y', x', z)\in A\colon xy, x'y'\in E(R_z)
    		\text{ but one of these pairs is not $\zeta_*$-connectable}\}\,.
\end{align*}

If $(x, y, y', x', z)\in Q_1$, then by clause~\ref{it:labs2} of Lemma~\ref{lem:28}
we have $d_{L_z}(y)\ge \tfrac{5n}{12}=\bigl(\tfrac13+\tfrac1{12}\bigr)n$ and due
to $|V(R_z)|\ge \tfrac23n$
it follows that $y$ is incident to at least $\tfrac n{12}$ edges in the link graph~$L_z$
running from $V\setminus V(R_z)$ to $V(R_z)$. Owing to condition~\ref{it:rc2}
from Proposition~\ref{prop:robust} there are for each $z\in V$ at most
$3\alpha n$ vertices $y$ with this property and, consequently, we have
\[
	|Q_1|\le 3\alpha n^5\,.
\]

Similarly if $(x, y, y', x', z)\in Q_2$, then at least one of the pairs
$xy$, $yy'$, or $x'y'$ connects~$V(R_z)$ to its complement in the link graph~$L_z$,
which shows
\[
	|Q_2|\le \tfrac64\alpha n^5\,.
\]

Moreover, the case $\zeta=\zeta_*$ of Fact~\ref{fact:easy} leads to
\[
	|Q_3|\le 2\zeta_* n^5\,.
\]
Finally, taking $1\gg \alpha, \zeta_*$ into account we get
\[
	|A\setminus (Q_1\cup Q_2\cup Q_3)|
	\ge \left(\frac 1{28^4}-\frac{9}{2}\alpha-2\zeta_*\right)n^5
	>\frac{n^5}{2^{20}}\,.
\]

Definition~\ref{dfn:absorbable} guarantees that for at most $\tfrac{n^5}{2^{21}}$ of the
quintuples
\[
	(x, y, y', x', z)\in A\setminus (Q_1\cup Q_2\cup Q_3)
\]
the vertex~$z$
can fail to be absorbable. Conversely every absorbable vertex can account for at most $n^4$
such quintuples. Thus there are indeed at least $\tfrac{n}{2^{21}}$ absorbable vertices.	
\end{proof}

It remains to consider the other part of our absorbers, i.e., the six hyperedges spanned by
$a$, $b$, $c$, $d$ together with $v$ and $z$ (see Figure~\ref{fig:absorber}).

\begin{lemma}\label{lem:abcd}
Given Setup~\ref{setup2}, for every vertex $v\in V$ there are at least $\alpha^4n^5$ quintuples $(a, b, c, d, z)\in V^5$
such that
\begin{enumerate}[label=\rmlabel]
\item\label{it:abcd1} $vab, vbc, vcd, zab, zbc, zcd\in E$,
\item\label{it:abcd2} $ab$ and $cd$ are $\zeta_*$-connectable,
\item\label{it:abcd3} and $z$ is absorbable.
\end{enumerate}
\end{lemma}

\begin{proof}
For every vertex $v\in V$ and every fixed absorbable vertex $z\in V$ Corollary~\ref{cor:34}
tells us $|E(R_v)\cap E(R_z)|\ge\alpha \frac{n^2}2$ and a result due to
Blakley and Roy~\cite{BR} (asserting the validity of Sidorenko's conjecture~\cites{ES82,Si84}
for paths) entails that there are at least $\alpha^3n^4$ quadruples $(a, b, c, d)\in V^4$
forming a three-edge walk in both graphs $R_v$ and $R_z$. Together with Lemma \ref{lem:number-abs}
this shows that there are at least $\frac{\alpha^3}{2^{21}}n^5$ quintuples
$(a, b, c, d, z)\in V^5$ satisfying properties~\ref{it:abcd1} and~\ref{it:abcd3} of Lemma~\ref{lem:abcd},
and $ab, cd\in E(R_v)\cap E(R_z)$ instead of property~\ref{it:abcd2}. As a consequence of
Fact~\ref{fact:easy}, however, there are at most $2\zeta_* n^5$ such quintuples
for which one of these two pairs fails to be $\zeta_*$-connectable.
As $1\gg \alpha \gg \zeta_*$ implies $\tfrac{\alpha^3}{2^{21}}-2\zeta_*>\alpha^4$,
the lemma follows.
\end{proof}

Lemma~\ref{lem:abcd} easily implies that there are $\Omega(n^9)$
$v$-absorbers for every vertex $v\in V$. In addition we can
also ensure that these absorbers are outside the reservoir~$\cR$.

\begin{lemma}\label{lem:count-absorbers}
Given Setup~\ref{setup2},
for every $v\in V$ the number of $v$-absorbers in $(V\setminus \cR)^9$
is at least $\alpha^5n^9$.
\end{lemma}

\begin{proof}
Combining Lemma~\ref{lem:abcd} with Definition \ref{dfn:absorbable},
we learn that there are at least $\tfrac{\alpha^4}{2^{21}}n^9$
$9$-tuples meeting all  requirements from that definition
except that some of the $10$ vertices $v, a, \ldots, x'$
might coincide. However, there can be at most~$45n^8$ such bad $9$-tuples.
Moreover, at most $9\theta_*^2n^9$ members of $V^9$ can use a vertex from
the reservoir and, consequently, the number of desired $v$-absorbers is at least
$\bigl(\tfrac{\alpha^4}{2^{21}}-\tfrac{45}{n}-9\theta_*^2\bigr)n\ge\alpha^5 n$.
\end{proof}

Having established that there are at least $\Omega(n^9)$ $v$-absorbers
with connectable pairs for every $v\in V$ we can build the absorbing path
by a  standard probabilistic argument. First we find a suitable selection
of $\Omega(n)$ disjoint $9$-tuples that contain many $v$-absorbers for every~$v$, which
is rendered by the following lemma. In a second step we utilise the
$\zeta_*$-connectable pairs and connect these $9$-tuples to the absorbing path avoiding the
reservoir set~$\cR$.

\begin{lemma}\label{lem:choose-F}
Given Setup~\ref{setup2}, there is a set $\cF\subseteq (V\setminus \cR)^9$ with the following properties:
\begin{enumerate}[label=\rmlabel]
\item\label{it:cF1} $|\cF|\le 8\alpha^{-5}\theta_*^{2}n$,
\item\label{it:cF2} all vertices of every $9$-tuple in $\cF$ are distinct and the 9-tuples in $\cF$ are pairwise disjoint,
\item\label{it:cF3} if $(a, b, c, d, z, x, y, y', x')\in \cF$, then
					$abz, bzc, zcd, xyy', yy'x'\in E$ and the pairs $ab, cd, xy, x'y'$
					are $\zeta_*$-connectable,
\item\label{it:cF4} and for every $v\in V$ there are at least $2\theta_*^2 n$ many
					$v$-absorbers in $\cF$.
\end{enumerate}
\end{lemma}

\begin{proof}
Set
\begin{equation*}\label{eq:gamma}
	\gamma=\frac{4\theta_*^2}{\alpha^5}
\end{equation*}
and consider a random selection $\cX\subseteq (V\setminus \cR)^9$ containing
each such $9$-tuple independently with probability $p=\gamma n^{-8}$.
Since $\EE[|\cX|]\le p n^9=\gamma n$, Markov's inequality yields
\begin{equation}\label{eq:X-big}
	\PP\big(|\cX|>2\gamma n\big)\le \frac12\,.
\end{equation}

Let us call two distinct $9$-tuples from $V^9$ {\it overlapping} if there is a vertex
occurring in both. Evidently, there are at most $81n^{17}$ ordered pairs of overlapping
$9$-tuples. Hence the random variable~$P$ giving the number of such pairs
both of whose components are in $\cX$ satisfies
\[
	\EE[P]\le 81n^{17}p^2=81\gamma^2n\,.
\]
By $\alpha\gg \theta_*$ we have $18 \gamma\le \theta_*$
and thus a further application of Markov's inequality discloses
\begin{equation}\label{eq:P-big}
	\PP\big(P>\theta_*^2n\big)\le \PP\big(P>324\gamma^2 n\big)\le \frac14\,.
\end{equation}

In view of Lemma~\ref{lem:count-absorbers} for each vertex $v\in V$
the set $\cA_v$ containing all $v$-absorbers from~$(V\setminus \cR)^9$ has the
property $\EE[|\cA_v\cap \cX|]\ge \alpha^5n^9p=\alpha^5\gamma n=4\theta_*^2n$.
Since $|\cA_v\cap \cX|$ is binomially distributed, Chernoff's inequality gives
for every $v\in V$
\begin{equation}\label{eq:Av-small}
	\PP\big(|\cA_v\cap \cX|\le 3\theta_*^2n\big)\le \exp\big(-\Omega(n)\big)<\frac 1{5n}
	\,.
\end{equation}

Owing to~\eqref{eq:X-big},~\eqref{eq:P-big}, and~\eqref{eq:Av-small} there is an
``instance'' $\cF_*$ of $\cX$ satisfying the following:
\begin{enumerate}
\item[$\bullet$] $|\cF_*|\le 2\gamma n$,
\item[$\bullet$] $\cF_*$ contains at most $\theta_*^2 n$ overlapping pairs,
\item[$\bullet$] and for every $v\in V$ the number of $v$-absorbers in $\cF_*$ is
				at least $3\theta_*^2n$.
\end{enumerate}

To obtain the desired set $\cF$ we delete from $\cF_*$
all $9$-tuples containing some vertex more than once,
all $9$-tuples belonging to an overlapping pair,
and all $9$-tuples violating~\ref{it:cF3}.
Then~\ref{it:cF1} is immediate from $|\cF|\le |\cF^*|$,~\ref{it:cF2} and~\ref{it:cF3}
hold by construction, and for~\ref{it:cF4} we recall that $v$-absorbers satisfy~\ref{it:cF3}
by definition.
\end{proof}

Finally, we are ready to build an absorbing path and thus  establish Proposition~\ref{prop:absorbingP}.

\begin{proof}[Proof of Proposition~\ref{prop:absorbingP}]
Let $\cF\subseteq (V\setminus \cR)^9$ be as obtained in Lemma~\ref{lem:choose-F}.
By condition~\ref{it:cF3} from this lemma for every $(a, b, c, d, z, x, y, y', x')\in \cF$
we may consider the tight paths $abzcd$ and $xyy'x'$. By~\ref{it:cF2} these paths
are mutually vertex-disjoint and by~\ref{it:cF1} the set~$\cG$ of all these paths
satisfies $|\cG|=2|\cF|\le 16\alpha^{-5}\theta_*^2n$.

Using the connecting lemma we will now prove that there is a path
$P_A\subseteq H-\cR$
\begin{enumerate}[label=\alabel]
\item\label{it:PAa} containing all members of $\cG$ as subpaths,
\item\label{it:PAb} whose end-pairs are $\zeta_*$-connectable,
\item\label{it:PAc} and whose length is at most $(3\ell+6)|\cG|$.
\end{enumerate}

Essentially, the reason why such a path exists is that starting with any member of $\cG$
we can construct $P_A$ by~$|\cG|-1$ successive applications of the connecting lemma
attaching one further path from~$|\cG|$ in each step. When carrying this plan out,
we need to avoid entering the reservoir and we need to be careful not to use the same
vertex multiple times.

To show that this is possible we consider a maximal subset $\cG^*\subseteq \cG$
such that some path $P^*_A\subseteq H-\cR$ has the
properties~\ref{it:PAa},~\ref{it:PAb}, and~\ref{it:PAc} enumerated above
with $\cG$ replaced by~$\cG^*$. As the end-pairs of members of $\cG$ are by definition
$\zeta_*$-connectable we have $P^*_A\ne\varnothing$.
From~\ref{it:PAc} and $1\gg \alpha, \ell^{-1}\gg \theta_*$ we infer
\begin{equation}\label{eq:PA-short1}
	|V(P^*_A)|\le 2+(3\ell+6)|\cG^*|\le 4\ell |\cG|
	\le 64\ell\alpha^{-5}\theta_*^2n\le \theta_*^{3/2}n
\end{equation}
and thus our upper bound on the size of the reservoir leads to
\begin{equation}\label{eq:PA-short2}
	|V(P^*_A)|+|\cR| \le 2\theta_*^{3/2}n \le \frac{\theta_*n}{2(3\ell+1)}\,.
\end{equation}
Assume for the sake of contradiction that $\cG^*\ne \cG$. Let $(z, w)$ be the ending pair
of $P^*_A$ and let~$P$ be an arbitrary path from $\cG\setminus \cG^*$ with starting pair $(x, y)$.
Since both~$(z,w)$ and~$(x,y)$ are $\zeta_*$-connectable,
Proposition~\ref{lem:con} tells us that there are at least $\theta_* n^{3\ell+1}$
tight $(z, w)$-$(x, y)$-paths of length $3(\ell+1)$. By~\eqref{eq:PA-short2} at least
half of these are internally disjoint from $V(P^*_A)\cup\cR$. In particular, there is at least one
such connection giving rise to a path $P^{**}_A\subseteq H-\cR$
starting with $P^*_A$, ending with $P$ and satisfying
\[
	|V(P^{**}_A)|=|V(P^{*}_A)|+(3\ell+1)+|V(P)|\le |V(P^{*}_A)|+(3\ell+6)
	\le 2+(3\ell+6)(|\cG^*|+1)\,.
\]
So $P^{**}_A$ exemplifies that $\cG^*\cup\{P\}$ contradicts the maximality of $\cG^*$
and this contradiction proves that we have indeed $\cG^*=\cG$, i.e., that a path
$P_A$ with the properties~\ref{it:PAa},~\ref{it:PAb}, and~\ref{it:PAc} promised above
does really exist.

As proved in~\eqref{eq:PA-short1} this path satisfies in particular the above
condition~\ref{it:PA1} of Proposition~\ref{prop:absorbingP}. Moreover,~\ref{it:PA2} is the same as~\ref{it:PAb}.
To finally establish~\ref{it:PA3} of Proposition~\ref{prop:absorbingP} one absorbs the up to at most $2\theta_*^2n$
vertices from $X$ one by one
into~$P_A$. By the discussion after Definition~\ref{dfn-vabs}
this is possible due to~\ref{it:PAa} combined with clause~\ref{it:cF4}
from Lemma~\ref{lem:choose-F}.
\end{proof}

\section{Almost spanning cycle}\label{sec:longpath}
This section is dedicated to the proof of
Proposition~\ref{prop:longpath}.
Most of the work we need to perform concerns
the construction of a long path $Q$ in the induced subhypergraph $\wh{H}=H-V(P_A)$
that covers ``almost all'' vertices, but leaves the reservoir set~$\cR$
``almost intact.'' Besides, the end-pairs of this path should be sufficiently connectable
so that it can easily be included into $C$.
These properties of $Q$ are made precise by the following statement.

\begin{lemma}\label{lem:long} Given Setup~\ref{setup2} as well as an absorbing path $P_A$
as provided by Proposition~\ref{prop:absorbingP},
there is a path $Q\subseteq \wh{H}=H-V(P_A)$ such that
\begin{enumerate}[label=\rmlabel]
\item\label{it:lp1} $|V(\wh{H})\setminus \bigl(\cR\cup V(Q)\bigr)|\le \theta_*^2 n$,
\item\label{it:lp2} $|V(Q)\cap \cR|\le \theta_{**}^2n$,
\item\label{it:lp3} and the end-pairs of~$Q$ are $\zeta_{**}$-connectable.
\end{enumerate}
\end{lemma}

Before we prove Lemma~\ref{lem:long}, we deduce Proposition~\ref{prop:longpath} from the lemma.

\begin{proof}[Proof of Proposition~\ref{prop:longpath}]
Given the path $Q\subseteq H-V(P_A)$ by Lemma~\ref{lem:long},
one simply connects the end-pairs of~$P_A$ with the end-pairs of $Q$
through ``free vertices'' from the reservoir using Lemma~\ref{lem:use-reservoir}.
The connectability assumption of that lemma is satisfied by condition~\ref{it:PA2}
from Proposition~\ref{prop:absorbingP} and by condition~\ref{it:lp3} from
Lemma~\ref{lem:long}. Each of these connections uses exactly $3\l+1$ vertices of $\cR$.
Consequently, it follows from Lemma~\ref{lem:long}\,\ref{it:lp2} that
at most $\theta_{**}^2n+(3\ell+1)<2\theta_{**}^2n$
vertices from $\cR$ need to be avoided and Lemma~\ref{lem:use-reservoir} applies.
The resulting tight cycle~$C$ contains
all but at most $\theta_*^2 n$ vertices from $V\setminus\cR$ (see Lemma~\ref{lem:long}\,\ref{it:lp1}).
Furthermore, since $|\cR|\leq \theta_*^2 n$ (see Setup~\ref{setup2} and Proposition~\ref{prop:reservoir})
it follows that $C$ covers all but at most $2\theta_*^2 n$ vertices as required by
Proposition~\ref{prop:longpath}.
\end{proof}

It remains to establish Lemma~\ref{lem:long}. This proof will occupy the
remainder of this section and, as explained in Section~\ref{sec:plan},
it completes the proof of our main result.
In the proof we make use of the following extension of the Erd\H os--Gallai theorem~\cite{ErGa59} concerning
the extremal problem for \emph{long} paths. We state the result of Faudree and Schelp~\cite{FS75}*{page~151}
in a form tailored for our application.
\begin{theorem}[Faudree and Schelp]
\label{thm:FS}
If $G=(V,E)$ is a graph not containing a path of length $\lambda|V|$ for $\lambda>1/2$, then
$|E|\leq \big(\lambda^2+(1-\lambda)^2\big)|V|^2/2$.	
\qed
\end{theorem}

\begin{proof}[Proof of Lemma~\ref{lem:long}]
We fix an integer $M$ satisfying the conditions
\begin{equation}\label{eq:M}
	\theta_{**}\gg \frac{1}{M}\gg \frac1n
	\qand
	M\equiv 2\pmod{3}\,.
\end{equation}

\begin{figure}[ht]
\centering
\begin{subfigure}[b]{0.50\textwidth}
\begin{tikzpicture}[scale=1.1]
	
	\def\r{0.7cm}
  	\def\v{1.4mm}
	\def\n{9}
	\def\con{
		($(72:\r)+(217:{\r+cos(217*\n)*\v})$)
			\foreach \a in {217,...,-97}{ -- ($(72:\r)+(\a:{\r+cos(\a*\n)*\v})$)}
	}
	
	\begin{pgfonlayer}{front}
	\draw[red!70!black, line width=2pt, rounded corners, line cap=round] (0,0) -- (2,0) -- (2,1);
	
	\begin{scope}[shift={(2.36,0.84)}]
	\draw[red!70!black, line width=2pt, line cap=round, rotate=30]
		\con;
	\end{scope}
	
	\draw[red!70!black, line width=2pt, rounded corners, line cap=round]
		(2.437,1) -- ++(0,-1) -- ++(2,0) -- ++(0,1);
	
	\begin{scope}[shift={(4.797,0.84)}]
	\draw[red!70!black, line width=2pt, line cap=round, rotate=30]
		\con;
	\end{scope}
	
	\draw[red!70!black, line width=2pt, rounded corners, line cap=round]
		(4.874,1) -- ++(0,-1) -- ++(2,0);
	\end{pgfonlayer}
	
	\draw[orange!70!yellow, line width=1.5pt, rounded corners] (1.1,0.6) rectangle (5.774,2.6);
	\fill[orange!70!yellow,opacity=0.2,line width=1.5pt, rounded corners] (1.1,0.6) rectangle (5.774,2.6);
	
	\node at (0.75,2.45) {\large \textcolor{orange!90!yellow}{$\cR$}};
	\node at (0.125,0.28) {\large \textcolor{red!70!black}{$Q$}};
	\node at (3.5,-1.2) {segments of $Q$};
	
	\draw[black!50!white, line width=1.5pt, line cap=round, ->] (2.15,-1.2) -- (1.0,-0.2);
	\draw[black!50!white, line width=1.5pt, line cap=round, ->] (3.5,-0.95) -- (3.5,-0.2);
	\draw[black!50!white, line width=1.5pt, line cap=round, ->] (4.82,-1.2) -- (5.9,-0.2);
	
	\end{tikzpicture}
	\caption{Segments of \textcolor{red!70!black}{$Q$} connected through \textcolor{orange!90!yellow}{$\cR$}}
	\label{sfig:Q}
\end{subfigure}\hfill
\begin{subfigure}[b]{0.47\textwidth}
\begin{tikzpicture}[scale=1.1]
	
	\draw[opacity=0] (0.2,0) -- (1,0);

	\def\piece{
		\coordinate (x1) at (0,0);
		\coordinate (x2) at (0.5,0);
		\coordinate (x3) at (2.0,0);
		\coordinate (x4) at (2.5,0);
		
		\draw[green!80!black, line width=2pt] (x1) -- (x2);
		\draw[line width=2pt, decorate, decoration={snake,segment length=7,amplitude=4, pre length=2pt, post length=2pt}, red!70!black] (x2) -- (x3);
		\draw[green!80!black, line width=2pt] (x3) -- (x4);	
		
		\begin{pgfonlayer}{front}
			\foreach \i in {1,...,4}
				\fill (x\i) circle (2pt);
		\end{pgfonlayer}
	}
	
	\begin{scope}[shift={(0.8,0)}]
		\piece;
	\end{scope}
	
	\begin{scope}[shift={(4.1,0)}]
		\piece;
	\end{scope}

	\coordinate (y1) at (3.7,0.5);
		
	\begin{pgfonlayer}{front}
		\fill (y1) circle (2pt);
	\end{pgfonlayer}

	\qedge{(y1)}{(3.3,0)}{(2.8,0)}{4pt}{1.5pt}{red!70!black}{red!70!black,opacity=0.2};
	\qedge{(y1)}{(4.1,0)}{(3.3,0)}{4pt}{1.5pt}{red!70!black}{red!70!black,opacity=0.2};
	\qedge{(y1)}{(4.6,0)}{(4.1,0)}{4pt}{1.5pt}{red!70!black}{red!70!black,opacity=0.2};

	\node at (3.7,-1.2) {pieces of a segment};
	\draw[black!50!white, line width=1.5pt, line cap=round, ->] (2.8,-0.95) -- (1.95,-0.3);
	\draw[black!50!white, line width=1.5pt, line cap=round, ->] (4.6,-0.95) -- (5.35,-0.3);

		\end{tikzpicture}
	\caption{Pieces with \textcolor{green!60!black}{$\zeta_{**}$-connectable ends}}
	\label{sfig:fragments}
    \end{subfigure}
\caption{Segments and pieces of the tight path $Q$.}
\label{fig:Qfragments}
\end{figure}
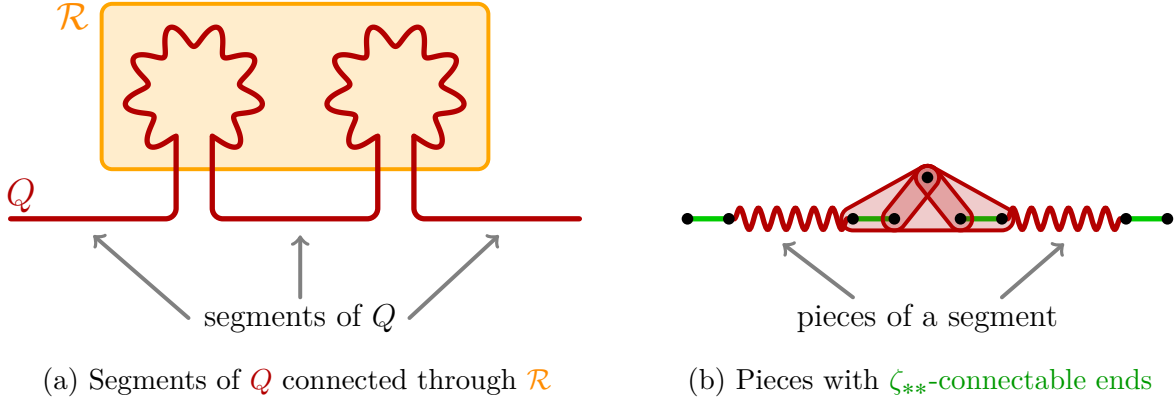

The desired path $Q$ will consist of many ``segments'' from~$\wh{H}- \cR$ that are connected
with each other through the reservoir~$\cR$ (see Figure~\ref{sfig:Q}).
For technical reasons it will be helpful to assume
that every segment $F$ satisfies
$$|V(F)|\equiv -1\pmod{M+1}$$ and that it has $\zeta_{**}$-connectable
end-pairs. The former property of these segments allows us to think of them as being composed
of several ``pieces'' consisting of $M$ vertices each, such
that any two consecutive pieces are connected with
each other through one further vertex (see Figure~\ref{sfig:fragments}).
These pieces will be taken from the set
\[
	\ccP=\bigl\{P\subseteq \wh{H}- \cR\colon P \text{ is an $M$-vertex tight path whose
	end-pairs are $\zeta_{**}$-connectable}\bigr\}.
\]
Roughly speaking, the strategy of the proof below is to show that a path $Q$ of the
kind just described will satisfy the conclusion of Lemma~\ref{lem:long}
as soon as it is ``maximal'' in the sense we will make precise next. To formulate
this maximality condition, it will be convenient to talk not only about the
path $Q$ itself but also about the set $\ccC\subseteq \ccP$ of pieces used in its
construction. We collect all properties that we require from the pair $(\ccC, Q)$
into the definition that follows.

\subsection*{Candidates}
A pair $(\ccC, Q)$ consisting of a subset $\ccC\subseteq \ccP$ whose members are mutually
vertex-disjoint and a tight path $Q\subseteq \wh{H}$ is said to be a {\it candidate} if
\begin{enumerate}[label=\alabel]
\item\label{it:C1} every $P\in \ccC$ is a subpath of $Q$,
\item\label{it:C2} if $P', P''\in \ccC$ with $P'\ne P''$ lie on $Q$ in such a way
	that no $P\in\ccC$ lies between them, then between~$P'$ and $P''$ there is
	\begin{enumerate}[label=\rmlabel]
	\item\label{it:C2i} either a single vertex
	\item\label{it:C2ii} or there are only vertices from $\cR$,
	\end{enumerate}
\item\label{it:C3} provided $\ccC\ne \emptyset$, the path $Q$ starts and ends
	with a path from $\ccC$,
\item\label{it:C4} and $|V(Q)\cap \cR|\le 19\alpha^{-1}\ell|\ccC|$.
\end{enumerate}

For instance, the pair consisting of the empty set and the empty path is a candidate.
Now let $(\ccC, Q)$ be a candidate with $|\ccC|$ as large as possible. Suppose we know
that the set
\[
	U=V(\wh{H})\setminus \bigl(\cR\cup V(Q)\bigr)=V\bigl(H - V(P_A) - \cR - V(Q)\bigr)
\]
of unused vertices outside the reservoir satisfies
\begin{equation}\label{eq:U-small}
	|U| \le \theta_*^2 n\,.
\end{equation}
We claim that then $Q$ would have all the desired properties.
Indeed by~\eqref{eq:U-small} it would
satisfy~\ref{it:lp1} of Lemma~\ref{lem:long}. Since the members of
$\ccC$ are mutually disjoint, we have $|\ccC|\le \tfrac{n}{M}$ so from~\ref{it:C4} and
$\theta_{**}, \ell^{-1}, \alpha\gg M^{-1}$ we get~\ref{it:lp2}.
Moreover~\ref{it:C3} implies part~\ref{it:lp3} of Lemma~\ref{lem:long}.

Hence, for the rest of the argument we assume that \eqref{eq:U-small} is false
and intend to derive a contradiction by constructing a ``better'' candidate $(\ccC',Q')$
with a larger family $\ccC'$.
Obviously, the path of such a candidate will need to contain some vertices from $U$
and to prepare ourself for a later stage of the argument we will now deal
with the connectability properties of the robust subgraphs of these vertices.
More precisely, for each $u\in U$ we define a subgraph $\overline{R}_u\subseteq R_u$
with the same set of vertices by deleting all edges that are not $\zeta_{**}$-connectable.
Owing to Fact~\ref{fact:easy} we have, in particular,
\[
	\sum_{u\in U} \bigl(e(R_u)-e(\overline{R}_u)\bigr)\le \frac{\zeta_{**}}{2}n^3\,.
\]
Consequently, the set
\[
	\Ubad=\big\{u\in U\colon e(\overline{R}_u)\le e(R_u)-\tfrac 18\alpha n^2\big\}
\]
satisfies $|\Ubad|\le 8\zeta_{**}\alpha^{-1}n$ and, by $\theta_*, \alpha\gg \zeta_{**}$,
this leads to
\begin{equation}\label{eq:N1-small}
	|\Ubad|\le \frac12 \theta_*^2 n\,.
\end{equation}
For each $u\in U\setminus \Ubad$ we introduce the real number
$\eta_u\in\bigl[\tfrac23+\tfrac{\alpha}{2}, 1\bigr]$ by
\begin{equation}\label{eq:eta-v}
	\big|V\big(\overline{R}_u\big)\big|=|V(R_u)|=\eta_u n
\end{equation}
and observe that part~\ref{it:rc2} from Proposition~\ref{prop:robust}  implies
\begin{equation}\label{eq:eta-e}
	e\big(\overline{R}_u\big)\ge
	e\big(R_u\big)-\frac18\alpha n^2\ge
	\left(\frac59+\frac{\alpha}4-(1-\eta_u)^2\right)\frac{n^2}{2}\,.
\end{equation}
\subsection*{Useful societies}
Let $B_1, \ldots, B_{|\ccC|}$ be the vertex sets
of the paths belonging to $\ccC$ and fix an arbitrary partition
\[
	U	=
	B_{|\ccC|+1}\dcup \ldots \dcup B_{\nu}\dcup B'\,,
\]
with
\[
	|B_{|\ccC|+1}|=\dots=|B_\nu|=M\qqand |B'|<M\,.
\]

The sets belonging to \[
	\ccB=\{B_1, \ldots, B_\nu\}
\]
will be referred to as {\it blocks}. The size of their union
\begin{equation}\label{eq:intro-B}
	B=B_1\dcup B_2\dcup \ldots \dcup B_\nu,
\end{equation}
 in view of candidacy property~\ref{it:C2}, can be bounded from below by
\[
	|B|=M\nu\ge n-|V(P_A)|-|\cR|-|\ccC|-|B'|\ge (1-\theta_*-\theta_*^2)n-\nu-M,
\]
where we used bounds on $|V(P_A)|$ and $|\cR|$ from Propositions~\ref{prop:absorbingP} and~\ref{prop:reservoir}.
Thus, observing that $\nu\leq n/M$ and recalling that $\theta_*\gg M^{-1}$, we obtain
\begin{equation}\label{eq:BBBB}
	|B|\ge (1-2\theta_*)n\,.
\end{equation}
Consequently, by~\eqref{eq:eta-v} and \eqref{eq:eta-e}, recalling that $\alpha\gg \theta_*$, for
every $u\in U\setminus \Ubad$
\begin{equation}\label{eq:elim-B}
	\frac{|V\big(\overline{R}_u\big)\cap B|}{M\nu}\le \frac{\eta_u}{1-2\theta_*}\le\eta_u+\frac{\alpha}{36}
	\qqand
	\frac{e_{\overline{R}_u}(B)}{M^2\binom{\nu}{2}}\ge
	\frac59+\frac{2\alpha}9-(1-\eta_u)^2,\end{equation}
where $e_G(A)$ stands for the number of edges in $G[A]$, the subgraph of $G$ induced by a subset of vertices $A\subseteq V(G)$.

A {\it society} is a set of $m$ blocks,
where
\begin{equation}\label{eq:m}
	m=1+\left\lceil\frac{36}{\alpha}\right\rceil\,.
\end{equation}
The collection of all $\binom{\nu}{m}$ societies will be denoted by $\gS$.

\begin{dfn}\label{dfn:use}
A society $\cS\in\gS$ is said to be {\it useful} for a vertex $u\in U\setminus \Ubad$
if for its union $S=\bigcup \cS$ and the real number $\tau$ defined by
$|S\cap V(\overline{R}_u)|=\tau |S|$,
\[
	e_{\overline{R}_u}\big(S\cap V(\overline{R}_u)\big)
	\ge
	\left(\frac59+\frac{\alpha}{9}-(1-\eta_u)(1+\eta_u-2\tau)\right)\frac{|S|^2}{2}\,.
\]
\end{dfn}

The following claim may explain the terminology used in Definition~\ref{dfn:use}.

\begin{claim}\label{clm:use-of-usefullness}
If a society $\cS\in\gS$ is useful for a vertex $u\in U\setminus \Ubad$, then
the graph $\overline{R}_u$ contains a graph  path on $\tfrac23(M+1)(m+6)$
vertices all of which belong to $S=\bigcup \cS$.
\end{claim}

\begin{proof}
Notice that by~\eqref{eq:M} the number $\tfrac23(m+6)(M+1)$ is indeed an integer.
Since
$\alpha m\ge 36+\alpha$, it follows from $\alpha\gg M^{-1}$ and~\eqref{eq:m} that
\[
	\frac{6(m+6)}{\alpha m-36}\le M\,,
\]
whence
\[
	\frac 23(M+1)(m+6)
	\le
	\frac 23 M\bigl((m+6)+(\alpha m-36)/6\bigr)
	=
	\left(\frac23 +\frac \alpha 9\right)Mm\,.
\]
Thus it suffices to find a path in the graph~$\overline{R}_u$ traversing
$\bigl(\tfrac23+\tfrac\alpha 9\bigr)Mm$ vertices all of which belong
to $S=\bigcup \cS$.
Let us define a
real number~$\rho$ by
\[
	e_{\overline{R}_u}\big(S\cap V(\overline{R}_u)\big)
	=
	\rho\frac{(Mm)^2}2
	=
	\rho\frac{|S|^2}2
	\,.
\]
Clearly $\tau^2\ge\rho$ and the definitions of $\tau$ and $\rho$ yield
\begin{equation}\label{eq:rhotau}
	e_{\overline{R}_u}\big(S\cap V(\overline{R}_u)\big)
	=
	\frac{\rho}{\tau^2}\cdot\frac{|S\cap V(\overline{R}_u)|^2}{2}\,.
\end{equation}
Moreover, the usefulness of $\cS$ implies
\begin{equation}\label{eq:tau2}
	\tau^2\ge \rho\ge \frac59+\frac{\alpha}{9}-(1-\eta_u)(1+\eta_u-2\tau)\,,
\end{equation}
which may be reformulated as
\begin{equation}\label{eq:tau3}
	\tau\left(\tau-\frac23\right)\ge \frac{\alpha}{9}
	+\left(\eta_u-\frac23\right)\left(\eta_u+\frac23-2\tau\right)\,.
\end{equation}
We have defined $\eta_u$ in \eqref{eq:eta-v} so that
$\eta_u\ge \tfrac23+\tfrac\alpha 2$. Consequently,
if $\tau< \tfrac23+\tfrac\alpha 9$, then
\[
	\frac{\alpha}{9}+\left(\eta_u-\frac23\right)\left(\eta_u+\frac23-2\tau\right)
	> \frac{\alpha}{9} > \tau\left(\tau-\frac23\right)\,,
\]
a contradiction with (\ref{eq:tau3}).
This proves that
\begin{equation}\label{eq:tau-big}
	\tau\ge \frac23+\frac\alpha 9\,.
\end{equation}
The right-hand side of~\eqref{eq:tau2} rewrites as
$\tfrac59+\tfrac{\alpha}{9}-(1-\tau)^2+(\eta_u-\tau)^2$
and for this reason we have
\begin{equation}\label{eq:tau-lessbig}
	\rho\ge \frac59+\frac{\alpha}{9}-(1-\tau)^2\,. 	
\end{equation}
Owing to~\eqref{eq:tau-big}, we deduce from Theorem~\ref{thm:FS} (applied with $\lambda=(2/3+\alpha/9)/\tau$
to the induced subgraph of $\overline{R}_u$ on the set $S\cap V(\overline{R}_u)$ of size $\tau|S|=\tau mM$) that the failure of our claim would imply
\[
	\frac{\rho}{\tau^2}\cdot\frac{|S\cap V(\overline{R}_u)|^2}{2}
	\overset{\eqref{eq:rhotau}}{=}
	e_{\overline{R}_u}\big(S\cap V(\overline{R}_u)\big)
	\leq
	\left(\left(\frac{\frac{2}{3}+\frac{\alpha}{9}}{\tau}\right)^2+\left(1-\frac{\frac{2}{3}+\frac{\alpha}{9}}{\tau}\right)^2\right)\frac{|S\cap V(\overline{R}_u)|^2}{2}	\,.
\]
Consequently, we arrive at
\[
	\left(\frac23+\frac\alpha 9\right)^2+\left(\tau-\frac23-\frac\alpha 9\right)^2
	\ge \rho
	\overset{\eqref{eq:tau-lessbig}}{\ge}
	\frac59+\frac{\alpha}{9}-(1-\tau)^2\,,
\]
whence
\[
	\left(\frac23+\frac\alpha 9\right)^2+\left(\frac13-\frac\alpha 9\right)^2
	\ge
	\frac59+\frac{\alpha}{9}\,,
\]
i.e.,
\[
	\frac{2}{27}\alpha+\frac{2}{81}\alpha^2\ge \frac 19\alpha\,,
\]
contrary to $1\gg\alpha$. This completes the proof of Claim \ref{clm:use-of-usefullness}.
\end{proof}

A counting argument shows that there exists a society that is useful for many vertices.
\begin{claim}\label{claim:cS'}
There is a society $\cS'\in\gS$
useful for at least $\tfrac{\alpha}{18}|U\setminus \Ubad|$ vertices $u\in U\setminus \Ubad$.
\end{claim}

\begin{proof}
The claim follows by double counting from the
assertion that for every $u\in U\setminus \Ubad$ the number of useful societies
is at least~$\tfrac{\alpha}{18}|\gS|$, which we verify below. For that consider a vertex $u\in U\setminus \Ubad$.
Suppose that $\gamma\binom{\nu}{m}$ of all $|\gS|=\binom{\nu}{m}$ societies are useful for $u$.
For $i\in[\nu]$ set
\[
	|B_i\cap V(\overline{R}_u)|=\tau_i M
\]
and for all $i$ and $j$ with $1\le i<j\le\nu$ set
\[
	e_{\overline{R}_u}\big(B_i\cap V(\overline{R}_u), B_j\cap V(\overline{R}_u)\big)
	=\rho_{ij} M^2\,.
\]
By Definition~\ref{dfn:use}, if the society $\cS=\{B_1, \ldots, B_m\}$ is not useful for $u$,
then
\[
	\sum_{1\le i<j\le m}\rho_{ij}
	\leq
	\frac{e_{\overline{R}_u}(S\cap V(\overline{R}_u))}{M^2}
	<
	\left(\frac59+\frac{\alpha}{9}-(1-\eta_u^2)\right)\frac{m^2}2
	+(1-\eta_u)m\sum_{1\le i\le m}\tau_i\,.
\]
If $\cS$ is useful we still have the trivial bound
\[
	\sum_{1\le i<j\le m}\rho_{ij}\le \binom{m}{2}\,.
\]
Summing over all societies we infer
\begin{multline*}
	\binom{\nu-2}{m-2}\sum_{1\le i<j\le \nu}\rho_{ij}\le
 	\gamma\binom{\nu}{m}\binom{m}{2}
	 +\binom{\nu}{m}\left(\frac59+\frac{\alpha}{9}-(1-\eta_u^2)\right)\frac{m^2}2 \\
 	 +(1-\eta_u)m\binom{\nu-1}{m-1}\sum_{1\le i\le \nu}\tau_i\,.
\end{multline*}
Dividing by $\binom{\nu-2}{m-2}\binom{\nu}{2}=\binom{\nu}{m}\binom{m}{2}$ one learns
that the set $B$ introduced in~\eqref{eq:intro-B} satisfies
\[
	\frac{e_{\overline{R}_u}(B)-\sum_{1\le i\le \nu}e_{\overline{R}_u}(B_i)}{M^2\binom{\nu}{2}}
	\le \gamma+\frac{m}{m-1}\left(\frac59+\frac{\alpha}{9}-(1-\eta_u^2)\right)
	+2(1-\eta_u)\frac m{m-1}\frac{|V(\overline{R}_u)\cap B|}{M\nu}\,.
\]
Owing to
\[
	\frac{\sum_{1\le i\le \nu}e_{\overline{R}_u}(B_i)}{M^2\binom{\nu}{2}}
	\le
	\frac{\nu\binom{M}{2}}{M^2\binom{\nu}{2}}
	\le
	\frac{1}{\nu-1}
	\le
	\frac{2}{\nu}
	=
	\frac{2M}{|B|}
	\overset{\eqref{eq:BBBB}}{\le}
	\frac{3M}{n}
	\le
	\frac{\alpha}{108}
\]
and~\eqref{eq:elim-B} this yields
\[
	\frac59+\frac{2\alpha}{9}-(1-\eta_u)^2
	\le
	\gamma+\frac\alpha{108}
	+\left(1+\frac{1}{m-1}\right)\left(\frac59+\frac{\alpha}{9}-(1-\eta_u^2)
	+2(1-\eta_u)\left(\eta_u+\frac\alpha{36}\right)\right)\,,
\]
whence
\[
	\frac{\alpha}{9}\le \gamma+\frac\alpha{108}+\frac{\alpha}{18}(1-\eta_u)+\frac{1}{m-1}\,.
\]
Hence, the choice of $m$ in~\eqref{eq:m} and the bound $\eta_u>\tfrac23$ yield indeed that
$\gamma\ge\tfrac{\alpha}{18}$.
\end{proof}
For the rest of the proof let $\cS'$ from Claim~\ref{claim:cS'} be fixed.
By~\eqref{eq:N1-small} and the purported falsity of~\eqref{eq:U-small}
this means that the set
\[
	U'=\{u\in U\setminus \Ubad\colon \cS' \text{ is useful for } u\}
\]
satisfies $|U'|\ge \tfrac{\alpha\theta_*^2}{36}n$. Now we apply
Claim~\ref{clm:use-of-usefullness} to each $u\in U'$. Each time
the outcome may be regarded as a sequence of $\tfrac23(M+1)(m+6)$ distinct vertices
from the set $S'=\bigcup \cS'$.
Due to $|S'|=Mm$ there are no more than $(Mm)!$ such sequences and thus there is a set
$U''\subseteq U'$ with
\[
	|U''|\ge  \frac{\alpha\theta_*^2n}{36 (Mm)!}\ge \frac13(M+1)(m+6)-1
\]
such that all graphs $\overline{R}_u$ with $u\in U''$ contain a common path $W$
on $\tfrac23(M+1)(m+6)$ vertices.

\subsection*{Augmenting \texorpdfstring{$Q$}{\it Q}} Using the vertices of $W$ and $\tfrac13(M+1)(m+6)-1$
arbitrary vertices from $U''$ we obtain a tight path $T\subseteq (\wh{H}-\cR)$
with $|V(T)|=(M+1)(m+6)-1$ and every vertex of $T$
with a position divisible by 3 is a vertex from $U''$ (see Figure~\ref{fig:picT}).

\begin{figure}[ht]
\centering
\begin{tikzpicture}[scale=1.19]

	\foreach \i in {1,...,8}{
		\coordinate (x\i) at (235+5*\i:13);
		\begin{pgfonlayer}{front}
			\fill  (x\i) circle (2pt);
		\end{pgfonlayer}
	}
	\foreach \i in {9,...,12}{
		\coordinate (x\i) at (240+5*\i:13);
		\begin{pgfonlayer}{front}
			\fill  (x\i) circle (2pt);
		\end{pgfonlayer}
	}
	
	\coordinate (y1) at (-4,-10.2);
	\coordinate (y2) at (-2,-10.2);
	\coordinate (y3) at (0,-10.2);
	\coordinate (y4) at (4,-10.2);
	
	\begin{pgfonlayer}{front}
		\foreach \i in {1,...,4}{
			\draw[blue!75!black, very thick]  (y\i) circle (2pt);
			\fill[blue!75!white]  (y\i) circle (2pt);
		}	
	\end{pgfonlayer}
	
	\draw[green, line width=2pt]
		(x1) \foreach \i in {2,...,8}{ -- (x\i)};
	\draw[green!70!black, line width=2pt, decorate, decoration={snake,segment length=7,amplitude=4, pre length=3pt, post length=2pt}]
		(x8) -- (x9);
	\draw[green, line width=2pt]
		(x9) \foreach \i in {10,...,12}{ -- (x\i)};
		
	\draw[blue!75!black, line width=2pt] (0,-10.2) ellipse (4.7cm and 18pt);
	\fill[blue!75!white,opacity=0.2] (0,-10.2) ellipse (4.7cm and 18pt);
	
	\node at (2,-10.2) {\large\textcolor{blue!75!black}{$U''$}};
	
	\node at (-6.4,-10.6) {\large\textcolor{red!70!black}{$T$}};
	\node at (-6.4,-12.1) {\large\textcolor{green!70!black}{$W$}};
		
	\qedge{(y1)}{(x2)}{(x1)}{4.5pt}{1.5pt}{red!70!black}{red!70!black,opacity=0.2};
	\qedge{(y1)}{(x3)}{(x2)}{4.5pt}{1.5pt}{red!70!black}{red!70!black,opacity=0.2};
	\qedge{(y1)}{(x4)}{(x3)}{4.5pt}{1.5pt}{red!70!black}{red!70!black,opacity=0.2};
	
	\qedge{(y2)}{(x4)}{(x3)}{4.5pt}{1.5pt}{red!70!black}{red!70!black,opacity=0.2};
	\qedge{(y2)}{(x5)}{(x4)}{4.5pt}{1.5pt}{red!70!black}{red!70!black,opacity=0.2};
	\qedge{(y2)}{(x6)}{(x5)}{4.5pt}{1.5pt}{red!70!black}{red!70!black,opacity=0.2};
	
	\qedge{(y3)}{(x6)}{(x5)}{4.5pt}{1.5pt}{red!70!black}{red!70!black,opacity=0.2};
	\qedge{(y3)}{(x7)}{(x6)}{4.5pt}{1.5pt}{red!70!black}{red!70!black,opacity=0.2};
	\qedge{(y3)}{(x8)}{(x7)}{4.5pt}{1.5pt}{red!70!black}{red!70!black,opacity=0.2};
	
	\qedge{(y4)}{(x10)}{(x9)}{4.5pt}{1.5pt}{red!70!black}{red!70!black,opacity=0.2};
	\qedge{(y4)}{(x11)}{(x10)}{4.5pt}{1.5pt}{red!70!black}{red!70!black,opacity=0.2};
	\qedge{(y4)}{(x12)}{(x11)}{4.5pt}{1.5pt}{red!70!black}{red!70!black,opacity=0.2};

\end{tikzpicture}
\caption{Tight \textcolor{red!60!black}{path $T$} on the graph path \textcolor{green!60!black}{$W$} of
\textcolor{green!60!black}{$\zeta_{**}$-connectable pairs}.}
\label{fig:picT}
\end{figure}
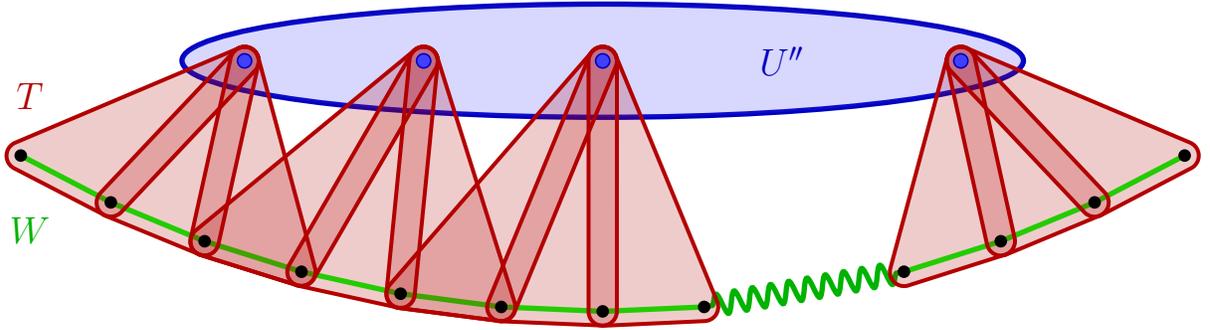

Next we split the path $T$ into $(m+6)$ tight paths $P_1,\dots,P_{m+6}$ on $M$ vertices each  such that
$T=P_1x_1P_2x_2\dots x_{m+5}P_{m+6}$ for some $x_1, \ldots, x_{m+5}\in V$.
In fact, owing to~\eqref{eq:M} the vertices $x_1, \ldots, x_{m+5}$ have a position divisible by 3
on the path $T$ and, therefore, they belong to $U''$.
Consequently, the end-pairs of the
paths $P_1, \ldots, P_{m+6}$ consist of consecutive vertices from $W$ and, hence, they
are $\zeta_{**}$-connectable. (This is the reason why we passed from the graphs $R_u$ to
the graphs $\overline{R}_u$ in the beginning of the argument.) In other words,
\[
	P_1, \ldots, P_{m+6}\in\ccP\,.
\]
Now let $\ccC^-$ be the collection of those paths from $\ccC$ whose vertex sets
belong to the society~$\cS'$, i.e., the paths from $\ccC^-\subseteq \ccC$ are blocks from the society $\cS'$,
and put
\[
	\ccC^0=(\ccC\setminus \ccC^-)\cup \{P_1, \ldots, P_{m+6}\}\,.
\]
Since $|\ccC^-|\le |\cS'|=m$, we have $|\ccC^0|\ge |\ccC|+6$ and thus to derive the
desired contradiction it is enough to construct a path $Q^0$ such that $(\ccC^0, Q^0)$
is a candidate. The idea for doing so is to take the subpaths into which the removal
of~$\ccC^-$ splits $Q$ as well as the path $T$ and to connect all of them by means of
Lemma~\ref{lem:use-reservoir}. Of course we may also need to remove several vertices
of the type mentioned in condition~\ref{it:C2}\,\ref{it:C2i} and in case $\ccC^-$
should contain the initial or terminal part of $Q$ we might also need to disregard some
$\cR$-vertices in order to achieve that $Q^0$ satisfies~\ref{it:C3}. The things that remain
to be checked are
\begin{enumerate}[label=\nlabel]
\item\label{it:cand:1} that we still have enough space in the reservoir to create the desired connections by applications of Lemma~\ref{lem:use-reservoir}
\item\label{it:cand:2} and that the new pair $(\ccC^0, Q^0)$ will again obey condition~\ref{it:C4}.
\end{enumerate}

Since $|\ccC^-|\le m$ at most $m+1$ successive
applications of Lemma~\ref{lem:use-reservoir} are required to connect all
pieces for building~$Q^0$.
Since $(\ccC, Q)$ satisfies~\ref{it:C4}, we know
\[
	|V(Q)\cap \cR|\le \frac{19\ell|\ccC|}{\alpha}\le\frac{19\ell n}{\alpha M}
	\le \theta_{**}^2n-(3\ell+1)m
\]
and, hence, there arises no problem with~\ref{it:cand:1}.

Utilising the same condition~\ref{it:C4}
more carefully we obtain
\[
	|V(Q^0)\cap \cR|\le |V(Q)\cap \cR|+(3\ell+1)(m+1)\leq \frac{19\ell|\ccC|}{\alpha}+(3\ell+1)(m+1)\,.
\]
So our choice of $m$ in~\eqref{eq:m} and $1\gg \alpha, \ell^{-1}$
lead to
\[
	|V(Q^0)\cap \cR|\le \frac{19\ell|\ccC|}{\alpha}+(3\ell+1)\left(\frac{36}{\alpha}+3\right)
	\le
	\frac{19\ell}{\alpha}(|\ccC|+6) \,.
\]
In the light of $|\ccC^0|\ge |\ccC|+6$ this shows that $(\ccC^0, Q^0)$ obeys condition~\ref{it:C4} and, hence, it
is indeed a candidate. As it contradicts the maximality of $(\ccC^0, Q^0)$
we have thereby proved the validity of~\eqref{eq:U-small} and as said above
Lemma~\ref{lem:long} is thereby proved as well.
\end{proof}


\begin{bibdiv}
\begin{biblist}

\bib{berge}{book}{
   author={Berge, Claude},
   title={Graphs and hypergraphs},
   note={Translated from the French by Edward Minieka;
   North-Holland Mathematical Library, Vol. 6},
   publisher={North-Holland Publishing Co., Amsterdam-London; American
   Elsevier Publishing Co., Inc., New York},
   date={1973},
   pages={xiv+528},
   review={\MR{0357172 (50 \#9640)}},
}

\bib{Bermond}{article}{
   author={Bermond, J.-C.},
   author={Germa, A.},
   author={Heydemann, M.-C.},
   author={Sotteau, D.},
   title={Hypergraphes hamiltoniens},
   language={French, with English summary},
   conference={
      title={Probl\`emes combinatoires et th\'eorie des graphes},
      address={Colloq. Internat. CNRS, Univ. Orsay, Orsay},
      date={1976},
   },
   book={
      series={Colloq. Internat. CNRS},
      volume={260},
      publisher={CNRS, Paris},
   },
   date={1978},
   pages={39--43},
   review={\MR{539937 (80j:05093)}},
}



\bib{BHS}{article}{
   author={Bu{\ss{}}, Enno},
   author={H{\`a}n, Hi{\d{\^e}}p},
   author={Schacht, Mathias},
   title={Minimum vertex degree conditions for loose Hamilton cycles in
   3-uniform hypergraphs},
   journal={J. Combin. Theory Ser. B},
   volume={103},
   date={2013},
   number={6},
   pages={658--678},
   issn={0095-8956},
   review={\MR{3127586}},
   doi={10.1016/j.jctb.2013.07.004},
}

\bib{BR}{article}{
   author={Blakley, G. R.},
   author={Roy, Prabir},
   title={A H\"older type inequality for symmetric matrices with nonnegative
   entries},
   journal={Proc. Amer. Math. Soc.},
   volume={16},
   date={1965},
   pages={1244--1245},
   issn={0002-9939},
   review={\MR{0184950}},
}

\bib{CoMy}{article}{
   author={Cooley, Oliver},
   author={Mycroft, Richard},
   title={The minimum vertex degree for an almost-spanning tight cycle in a
   3-uniform hypergraph},
   journal={Discrete Math.},
   volume={340},
   date={2017},
   number={6},
   pages={1172--1179},
   issn={0012-365X},
   review={\MR{3624602}},
}





\bib{CM}{article}{
   author={Czygrinow, Andrzej},
   author={Molla, Theodore},
   title={Tight codegree condition for the existence of loose Hamilton
   cycles in 3-graphs},
   journal={SIAM J. Discrete Math.},
   volume={28},
   date={2014},
   number={1},
   pages={67--76},
   issn={0895-4801},
   review={\MR{3150175}},
   doi={10.1137/120890417},
}


\bib{dirac}{article}{
   author={Dirac, G. A.},
   title={Some theorems on abstract graphs},
   journal={Proc. London Math. Soc. (3)},
   volume={2},
   date={1952},
   pages={69--81},
   issn={0024-6115},
   review={\MR{0047308 (13,856e)}},
}

\bib{ErGa59}{article}{
   author={Erd{\H{o}}s, P.},
   author={Gallai, T.},
   title={On maximal paths and circuits of graphs},
   language={English, with Russian summary},
   journal={Acta Math. Acad. Sci. Hungar},
   volume={10},
   date={1959},
   pages={337--356},
   issn={0001-5954},
   review={\MR{0114772}},
}

\bib{ES82}{article}{
   author={Erd{\H{o}}s, P.},
   author={Simonovits, M.},
   title={Cube-supersaturated graphs and related problems},
   conference={
      title={Progress in graph theory},
      address={Waterloo, Ont.},
      date={1982},
   },
   book={
      publisher={Academic Press, Toronto, ON},
   },
   date={1984},
   pages={203--218},
   review={\MR{776802}},
}
	

\bib{FS75}{article}{
   author={Faudree, R. J.},
   author={Schelp, R. H.},
   title={Path Ramsey numbers in multicolorings},
   journal={J. Combinatorial Theory Ser. B},
   volume={19},
   date={1975},
   number={2},
   pages={150--160},
   review={\MR{0412023}},
}


\bib{GPW}{article}{
   author={Glebov, Roman},
   author={Person, Yury},
   author={Weps, Wilma},
   title={On extremal hypergraphs for Hamiltonian cycles},
   journal={European J. Combin.},
   volume={33},
   date={2012},
   number={4},
   pages={544--555},
   issn={0195-6698},
   review={\MR{2864440}},
   doi={10.1016/j.ejc.2011.10.003},
}






\bib{YiJie}{article}{
   author={Han, Jie},
   author={Zhao, Yi},
   title={Minimum vertex degree threshold for loose Hamilton cycles in
   3-uniform hypergraphs},
   journal={J. Combin. Theory Ser. B},
   volume={114},
   date={2015},
   pages={70--96},
   issn={0095-8956},
   review={\MR{3354291}},
   doi={10.1016/j.jctb.2015.03.007},
}



\bib{JLR00}{book}{
   author={Janson, Svante},
   author={{\L}uczak, Tomasz},
   author={Ruci{\'n}ski, Andrzej},
   title={Random graphs},
   series={Wiley-Interscience Series in Discrete Mathematics and
   Optimization},
   publisher={Wiley-Interscience, New York},
   date={2000},
   pages={xii+333},
   isbn={0-471-17541-2},
   review={\MR{1782847}},
   doi={10.1002/9781118032718},
}

\bib{KK}{article}{
   author={Katona, Gy. Y.},
   author={Kierstead, H. A.},
   title={Hamiltonian chains in hypergraphs},
   journal={J. Graph Theory},
   volume={30},
   date={1999},
   number={3},
   pages={205--212},
   issn={0364-9024},
   review={\MR{1671170 (99k:05124)}},
   doi={10.1002/(SICI)1097-0118(199903)30:3<205::AID-JGT5>3.3.CO;2-F},
}







\bib{KO}{article}{
   author={K{\"u}hn, Daniela},
   author={Osthus, Deryk},
   title={Loose Hamilton cycles in 3-uniform hypergraphs of high minimum
   degree},
   journal={J. Combin. Theory Ser. B},
   volume={96},
   date={2006},
   number={6},
   pages={767--821},
   issn={0095-8956},
   review={\MR{2274077 (2007h:05115)}},
   doi={10.1016/j.jctb.2006.02.004},
}









\bib{sur}{article}{
   author={R{\"o}dl, Vojt{\v{e}}ch},
   author={Ruci{\'n}ski, Andrzej},
   title={Dirac-type questions for hypergraphs---a survey (or more problems
   for Endre to solve)},
   conference={
      title={An irregular mind},
   },
   book={
      series={Bolyai Soc. Math. Stud.},
      volume={21},
      publisher={J\'anos Bolyai Math. Soc., Budapest},
   },
   date={2010},
   pages={561--590},
   review={\MR{2815614 (2012j:05008)}},
   doi={10.1007/978-3-642-14444-8\_16},
}

\bib{1112}{article}{
   author={R{\"o}dl, Vojt{\v{e}}ch},
   author={Ruci{\'n}ski, Andrzej},
   title={Families of triples with high minimum degree are Hamiltonian},
   journal={Discuss. Math. Graph Theory},
   volume={34},
   date={2014},
   number={2},
   pages={361--381},
   issn={1234-3099},
   review={\MR{3194042}},
   doi={10.7151/dmgt.1743},
}

\bib{RRSSz}{article}{
   author={R\"odl, Vojt\v ech},
   author={Ruci\'nski, Andrzej},
   author={Schacht, Mathias},
   author={Szemer\'edi, Endre},
   title={On the Hamiltonicity of triple systems with high minimum degree},
   journal={Ann. Comb.},
   volume={21},
   date={2017},
   number={1},
   pages={95--117},
   issn={0218-0006},
   review={\MR{3613447}},
}




\bib{rrs3}{article}{
   author={R{\"o}dl, Vojt{\v{e}}ch},
   author={Ruci{\'n}ski, Andrzej},
   author={Szemer{\'e}di, Endre},
   title={A Dirac-type theorem for 3-uniform hypergraphs},
   journal={Combin. Probab. Comput.},
   volume={15},
   date={2006},
   number={1-2},
   pages={229--251},
   issn={0963-5483},
   review={\MR{2195584 (2006j:05144)}},
   doi={10.1017/S0963548305007042},
}
		




\bib{3}{article}{
   author={R{\"o}dl, Vojt{\v{e}}ch},
   author={Ruci{\'n}ski, Andrzej},
   author={Szemer{\'e}di, Endre},
   title={Dirac-type conditions for Hamiltonian paths and cycles in
   3-uniform hypergraphs},
   journal={Adv. Math.},
   volume={227},
   date={2011},
   number={3},
   pages={1225--1299},
   issn={0001-8708},
   review={\MR{2799606 (2012d:05213)}},
   doi={10.1016/j.aim.2011.03.007},
}

\bib{Si84}{article}{
   author={Sidorenko, A. F.},
   title={Extremal problems in graph theory and functional-analytic
   inequalities},
   language={Russian},
   conference={
      title={Proceedings of the All-Union seminar on discrete mathematics
      and its applications (Russian)},
      address={Moscow},
      date={1984},
   },
   book={
      publisher={Moskov. Gos. Univ., Mekh.-Mat. Fak., Moscow},
   },
   date={1986},
   pages={99--105},
   review={\MR{930303}},
}

\bib{Zhao-sur}{article}{
   author={Zhao, Yi},
   title={Recent advances on Dirac-type problems for hypergraphs},
   conference={
      title={Recent trends in combinatorics},
   },
   book={
      series={IMA Vol. Math. Appl.},
      volume={159},
      publisher={Springer, [Cham]},
   },
   date={2016},
   pages={145--165},
   review={\MR{3526407}},
}



\end{biblist}
\end{bibdiv}
\end{document}